\documentclass[a4paper,11pt,oneside]{article}%
\pdfoutput=1

\usepackage{cancel}
\usepackage{geometry}
\usepackage{amsfonts}
\usepackage{graphicx}
\usepackage[font={footnotesize,it},width=0.85\textwidth]{caption}
\usepackage{subcaption}
\usepackage[english]{babel}
\usepackage[latin1]{inputenc}
\usepackage[bbgreekl]{mathbbol}
\usepackage{bm}
\usepackage{bbm}
\usepackage{rotating}
\usepackage{tikz}
\usepackage{graphics}
\usepackage{times}
\usepackage[T1]{fontenc}
\usepackage{amssymb}
\usepackage{amsmath}
\usepackage{amsthm}
\usepackage{mathrsfs}
\usepackage{latexsym}
\usepackage{mathtools}
\usepackage{euscript}
\usepackage{eufrak}
\usepackage{amscd}
\usepackage{dsfont}
\usepackage[inline]{enumitem}
\usepackage{cancel}
\usepackage[nice]{nicefrac}
\usepackage[toc,page]{appendix}
\usepackage{changepage}
\usepackage[obeyDraft,textsize=footnotesize,colorinlistoftodos,bordercolor=orange]{todonotes}
\usepackage{tensor}
\usepackage{stmaryrd}
\usepackage[draft,authormarkup=none]{changes}
\definecolor{airforceblue}{rgb}{0.36, 0.54, 0.66}
\definecolor{cadmiumgreen}{rgb}{0.0, 0.42, 0.24}

\definechangesauthor[name={Giovanni Zanco}, color=blue]{GZ}
\definechangesauthor[name={Enrico Priola},color=green]{EP}
\definechangesauthor[name={Franco Flandoli},color=red]{FF}

\providecommand{\U}[1]{\protect\rule{.1in}{.1in}}

\makeatletter
\newtheorem*{rep@theorem}{\rep@title}
\newcommand{\newreptheorem}[2]{%
\newenvironment{rep#1}[1]{%
 \def\rep@title{#2 \ref{##1}}%
 \begin{rep@theorem}}%
 {\end{rep@theorem}}}
\makeatother

\makeatletter
\let\@@pmod\pmod
\DeclareRobustCommand{\pmod}{\@ifstar\@pmods\@@pmod}
\def\@pmods#1{\mkern4mu({\operator@font mod}\mkern 6mu#1)}
\makeatother

\makeatletter
\newcommand{\doublewidetilde}[1]{{%
  \mathpalette\double@widetilde{#1}%
}}
\newcommand{\double@widetilde}[2]{%
  \sbox\z@{$\m@th#1\widetilde{#2}$}%
  \ht\z@=.9\ht\z@
  \widetilde{\box\z@}%
}
\makeatother

\newtheorem{theorem}{Theorem}[section]
\newtheorem*{theorem*}{Theorem}

\newtheorem{definition}[theorem]{Definition}
\newtheorem*{definition*}{Definition}

\newtheorem{lemma}[theorem]{Lemma}

\theoremstyle{definition}

\newtheorem{remark}[theorem]{Remark}

\newreptheorem{definition}{Definition}

\newcounter{todocounter}

\DeclareMathOperator{\diag}{diag}
\DeclareMathOperator{\dist}{dist}

\newcommand{\bE}{\mathds{E}}

\newcommand{\bN}{\mathds{N}\,}
\newcommand{\bP}{\mathds{P}}
\newcommand{\bQ}{\mathds{Q}\,}
\newcommand{\bR}{\mathds{R}}

\newcommand{\bT}{\mathds{T}\,}

\newcommand{\bZ}{\mathds{Z}\,}

\newcommand{\rC}{\mathscr{C}}

\newcommand{\rL}{\mathscr{L}}

\newcommand{\sA}{\mathcal{A}}
\newcommand{\sB}{\mathcal{B}}
\newcommand{\sC}{\mathcal{C}}

\newcommand{\sE}{\mathcal{E}}
\newcommand{\sF}{\mathcal{F}}
\newcommand{\sG}{\mathcal{G}}
\newcommand{\sH}{\mathcal{H}}

\newcommand{\sJ}{\mathcal{J}}
\newcommand{\sK}{\mathcal{K}}
\newcommand{\sL}{\mathcal{L}}

\newcommand{\sS}{\mathcal{S}}
\newcommand{\sT}{\mathcal{T}}

\newcommand{\sW}{\mathcal{W}}

\newcommand{\sZ}{\mathcal{Z}}

\newcommand{\ud}{\,\mathrm{d}}

\newcommand{\ind}{\mathds{1}}

\newcommand{\VV}{\Vert}

\makeatletter
\newcommand{\can}[1]{%
  \sbox\z@{#1}%
  \mbox{%
    \kern.05em
    \can@rule{.3ex}%
    \can@rule{.6ex}%
    \kern.05em
    #1%
    \kern.1em
  }%
}
\newcommand{\can@rule}[1]{%
  \rlap{\vrule height \dimexpr#1+.2pt\relax
               depth -\dimexpr#1-.2pt\relax
               width \dimexpr\wd\z@+.1em\relax
  }%
}
\makeatother

\begin{document}

\author{Franco Flandoli
  \\ {\small Scuola Normale Superiore,}
\\ {\small Piazza dei Cavalieri 7,   Pisa,   Italy }
\\ \\
Enrico Priola 
 \\ {\small Dipartimento di Matematica, Universit\`a di Torino}
 \\ {\small via Carlo Alberto 10,   Torino,  Italy} \\
\\ 
Giovanni Zanco 
\\ {\small Institute of Science and Technology Austria (IST Austria)}
\\ {\small Am Campus 1, Klosterneuburg, Austria}
 \footnote{E-mail: flandoli@dma.unipi.it, enrico.priola@unito.it,   
 giovanni.zanco@ist.ac.at  The second author has been partially supported by INdAM through the GNAMPA Research Project (2017)  ``Sistemi stocastici singolari: buona posizione e problemi di controllo''. 
}
}

\title{A mean-field model with discontinuous coefficients for neurons with spatial interaction}

\vskip 2mm
\date{\small 23 January 2018}

\maketitle

\begin{abstract}
Starting from a microscopic model for a system of neurons evolving in time  which individually follow a stochastic integrate-and-fire type model, we study a mean-field limit of the system. Our model is described by a system of SDEs with discontinuous coefficients for the action potential of each neuron and takes into account the (random) spatial configuration of neurons  allowing the interaction to depend on it. In the limit as the number of particles tends to infinity, we obtain a nonlinear Fokker-Planck type PDE in two variables, with derivatives only with respect to one variable and discontinuous coefficients.
We also study strong well-posedness of the system of SDEs and prove the existence and uniqueness of a weak measure-valued solution to the PDE, obtained as the limit of the laws of the empirical measures for the system of particles.
\end{abstract}

{\vskip 1mm }


\section{Introduction}
\label{section:intro}
We propose a model for the action potential of $N$ neurons, with positions fixed in time, that follow integrate-and-fire type dynamics subject to noise and interact with each other through their spikes. The interaction we consider depends also on the positions of the neurons and is of mean-field type. Therefore, in the limit as $N$ tends to infinity each neuron interacts with infinitely many other neurons.\\
The presence of noise in the neuronal dynamics is experimentally confirmed and has been considered by various authors (see the monographs \cite{GKneu02}, \cite{Tuck88}). Integrate-and-fire (IF) models describe a simplified dynamics in which such effect can be studied in detail. Considering large networks of interacting neurons, each one having a membrane potential that evolves following a IF dynamic, leads to modeling the mean-firing rate of the network as the solution to a nonlinear partial differential equation that, at least for mean-field type interactions, is of Fokker-Planck type.\\
Fokker-Planck PDEs for neural networks have been studied recently in \cite{OBH09}, \cite{FB09}, \cite{CCP11}, \cite{CP14}, \cite{DIRT15}, based on an IF model for the potential of each neuron given in \cite{LR03}. As pointed out in \cite{DIRT15}, not much attention is paid in the literature to how the Fokker-Planck PDE is obtained; in particular one expects that the empirical measures of a network with $N$ neurons converge as $N\to\infty$ to the solution of the PDE. This has been rigorously shown only in \cite{DIRT16}, proving convergence to a McKean-Vlasov stochastic differential equation, and in \cite{DMGLP15}, where the hydrodynamic limit is considered.\\
The Fokker-Planck PDE obtained in the cited works exhibits blow-up in finite time, thus there is no global well-posedness, for certain ranges of parameters, due to the interaction term.\\
The model we propose here is simpler but it incorporates two additional aspects: a refractory period after the spike and a localized version of the interaction term, that is, an explicit dependence of the interaction on the positions of the neurons. The refractory period accounts for the fact that after emitting a spike, each neuron is inhibited from interaction. The dependence on a space variable allows to precisely prescribe the interation between different parts of the network; it can also describe the subdivision of the network in sub-populations, whose interaction with each other is of particular interest in neuroscience (see e.g. \cite{KMS96}, \cite{SDG17}). This leads to a description of finite speed signal propagation along the network.\\
More precisely, our mean-field interaction term has two main features: first, it depends both on the positions and on the voltage of the neurons, unlike many models available in the literature; second, it contains indicator functions of suitable intervals in $\bR$, thus requiring us to study a system of SDEs and a Fokker-Planck type PDE with irregular coefficients and dependence on the positions of the neurons that we treat as stochastic parameters.\\
We allow for great generality in the choice of the law of the positions of the neurons, only requiring finite first moment. Hence one can prescribe the geometry of the neural network choosing the law accordingly.\\
We study the limit behaviour of the empirical measures of the network and prove that the limit measure-valued function is the unique weak solution to a nonlinear PDE of Fokker-Planck type and that it exists up to any fixed time $T$, thus not exhibiting blow-up. From the technical point of view, to study the limit of the empirical measures we will also use some ideas of \cite{Oel89}.\\
Our model includes discontinuous coefficients, and is therefore a first step in the study of stochastic interacting particle systems with irregular coefficients. Some of the results we obtain can be immediately generalized to the case of SDEs with measurable and bounded coefficients, but we are able to study the limit PDE only when the coefficients are discontinuous on a set with $0$ Lebesgue measure (see, in particular, Lemma \ref{lemma Phi}). Therefore we stick in the main part of the paper to the particular coefficients coming from the model, and mention some possible generalizations in  Appendix.\\ 

The potential $V$ of each neuron is modeled, as a function of time, with a stochastic differential equation whose solution is projected on a torus given by the interval $[0,2]$ with the identification $0\equiv 2$. This choice allows to model the cycle of spikes of each neuron as we describe below. It is important to notice that, similarly to what is done in most IF models, we do not give a precise description of the spike phenomenon, but we model only the charging phase from the resting potential $v_R=0$ to the firing threshold $v_F=1$ and the refractory period after the spike; moreover we assume that there are no external input currents.

\smallskip
Consider a single cycle, that is $0\leq V\leq 2$. As $0\leq V<1$ the neuron charges, subject to spikes by nearby neurons (i.e., to interaction), to randomness and to the effect of discharge with constant rate (that corresponds to the fact that if no spikes happen in the connected neurons then some charge is lost as time passes); when $V$ reaches the threshold value $1$ the neuron fires and emits a spike into the network. On a real neuron this would have two effects: the potential would rapidly decrease below $0$ and then be restored to $0$, and the neuron would be at rest, inhibited from interacting and spiking for a small amount of time (the refractory period). We model this effect ``switching off'' the interaction term when $V>1$ and letting $V$ evolve as $\ud V=\ud t$ until it reaches the value $2$, where it is restored to $0$ (through the equivalence relation that defines the torus) and the charging cycle begins again. Therefore the values of $V$ between $1$ and $2$ do not correspond to a real life situation but are only a tool we resort to in order to have a convenient mathematical description of the phenomenon.

\smallskip 
To consider the interaction between $N$ neurons we deal with three factors  (see also equation (\ref{eq:dV}) below). Indeed if we consider the voltage $V^{i,N}$ and position $X^{i}$ of the $i$-th neuron, following the description above, a factor $\theta(X^i,X^j)$ accounts for the neuron being connected to some of the other neurons with positions $(X^{j})_j$; a factor $\ind_{[0,1]}(V^i)$ is due to the fact that the neuron feels the interaction only if it is in the charging phase; finally a factor $\ind_{[1,1+\delta]}(V^j)$ is due to the fact that the interaction considers contributions to the charging process only from neurons that have just had their spike $(\delta \in (0,1))$.
The choice of the values $0$, $1$ and $2$ is completely arbitrary, and is just used for our mathematical description; we also do not specify explicitly the form of some of the functions involved, since we only need to make assumptions on their regularity.

\smallskip
A possible more accurate model of the inhibition phase could require that also the noise term be switched off during the refractory period, that is, in our setting, as $V$ becomes larger than $1$. We are forced to include a small noise also in the inhibition phase, for mathematical reasons (i.e., we need $\epsilon$ below to be strictly positive, see in particular Theorems \ref{thm:SDEE1} and \ref{thm:xd} and Lemma \ref{lem:51}).
The effect of oscillations due to the noise at the transition between the active phase and the inhibition phase appears to be negligible on macroscopic scales, thanks to the drift (see for example figure \ref{fig:1b}).
On the other hand, the analysis of a model in which noise contributes only to the charging phase is mathematically extremely interesting, and we will face it in a future work.\\

\smallskip
Now we will introduce precisely the equations describing the model and will give an account of our results and of the following sections. {We also include some figures obtained simulating our model for a finite number of interacting neurons, showing that, even if simple, the model we propose gives a realistic description of single neurons and networks.}
\subsection{The model}
For a Borel set $A$ in an euclidean space, we will denote by $\rL_A$ the Lebesgue measure restricted to $A$.\\
Let $(\Omega,\sF,\bP)$ be a probability space, let $D$ be an open connected  domain in $\bR^3$  and $[0,T]\subset\bR$ a time interval. The microscopic model is as follows: for each $N\in\bN$ consider $N$ neurons, each identified by
\begin{enumerate}[label=\roman{*}.]
\item its position $X^{i}_0=\xi^i$, where $\xi^{i}$, $i\in\bN$ are i.i.d. random variables with finite first moment and such that $\forall i$ $\bP\left(\xi^{i}\in D\right)=1$. We denote by $\nu$ the law of each $\xi^i$. Since the neurons do not move, their position is modeled by the system of trivial equations
  \begin{equation}
    \label{eq:dX}
      \ud X^{i}_t=0\ ,\quad X^{i}_0=\xi^i\ \text{;}
  \end{equation}
\item its action potential given by $V^{i,N}_t\pmod{2}\in[0,2)$, where $V^{i,N}_t\in\bR$ is the strong solution to
  \begin{multline}
    \label{eq:dV}
      \ud V^{i,N}_t=\lambda\left(V^{i,N}_t\pmod{2}\right)\ud t\\+\frac{1}{N}\sum_{j=1}^N\theta\left({\xi^i},{\xi^j}\right)\ind_{[1,1+\delta]}\left(V^{j,N}_t\pmod{2}\right)\ind_{[0,1]}\left(V^{i,N}_t\pmod{2}\right)\ud t\\+\sigma^\epsilon\left(V^{i,N}_t\pmod{2}\right)\ud B^{i}_t\ ,
  \end{multline}
with initial condition $V^{i,N}_0=\eta^{i}\in[0,2)$. We assume that all random variables $\eta^i$ are i.i.d with law $\tilde \rho_0\rL_{[0,2)}$ and $\tilde \rho_0  \in L^2(0,2)$. Moreover, we assume that $\{ \xi^1, ..., \xi^N, \eta^1, ... \eta^N \}$ are independent for any $N\in\bN$.\\

The functions appearing above are given by
\begin{equation*}
\lambda(v)=-\hat\lambda v \ind_{[0,1]}(v)+\ind_{(1,2)}(v),\ \text{ with } \hat\lambda>0\ ;
\end{equation*}
\begin{equation*}
  \theta (x,y) \text{ is a bounded uniformly continuous function on }D\times D\ ;
\end{equation*}
\begin{gather*}  
\sigma^\epsilon(v) \text{ is a  $C_b^1([0,2])$-function  such that \;}
\sigma^\epsilon\geq \sqrt{2\epsilon}>0\ \\ \text{ and }\sigma^\epsilon(v)=\sqrt{2\epsilon}\; \text{ on }[1,2], \; \sigma^\epsilon(2)
=
\sigma^\epsilon(0)=\sqrt{2\epsilon},\\ 
\frac{d\sigma^\epsilon}{dx}(0)
=
\frac{d \sigma^\epsilon}{dx}(2)=0 \ , \text{ with }\epsilon\text{ fixed}.
\end{gather*}
For each $i\in\bN$ the processes $\left(B^{i}_t\right)$ are independent real-valued Brownian motions, independent of $\left(\xi^i\right)$ and $\left(\eta^i\right)$, and $\delta$ is a fixed real number in $(0,1)$.
\end{enumerate}
\begin{figure}[!tbp]
  \centering
\begin{subfigure}[b]{0.49\textwidth}
\includegraphics[width=\textwidth]{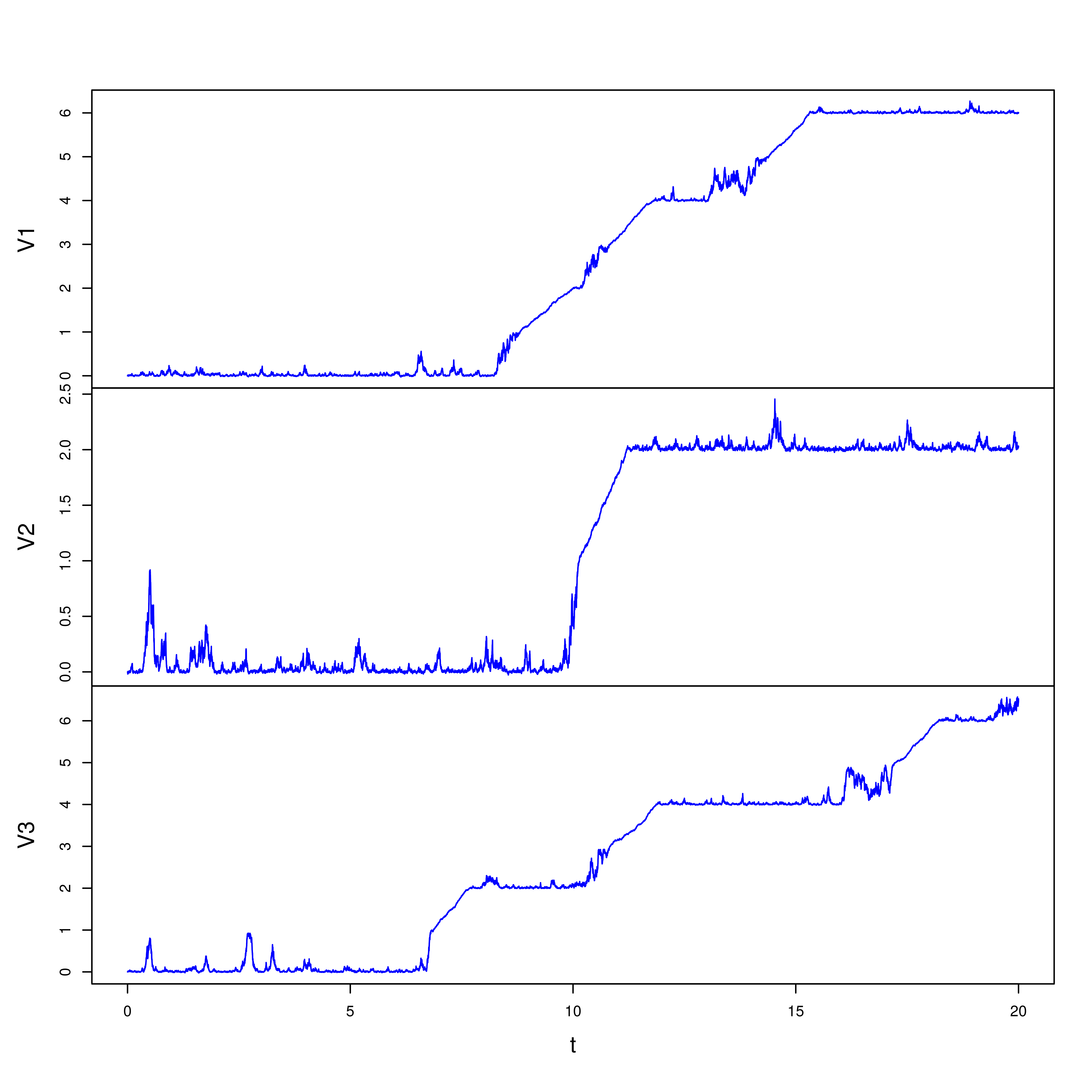}
\caption{Solutions $V^{i,3}$, $i=1,2,3$}
\label{fig:1a}
\end{subfigure}
\hfill
\begin{subfigure}[b]{0.49\textwidth}
  \includegraphics[width=\textwidth]{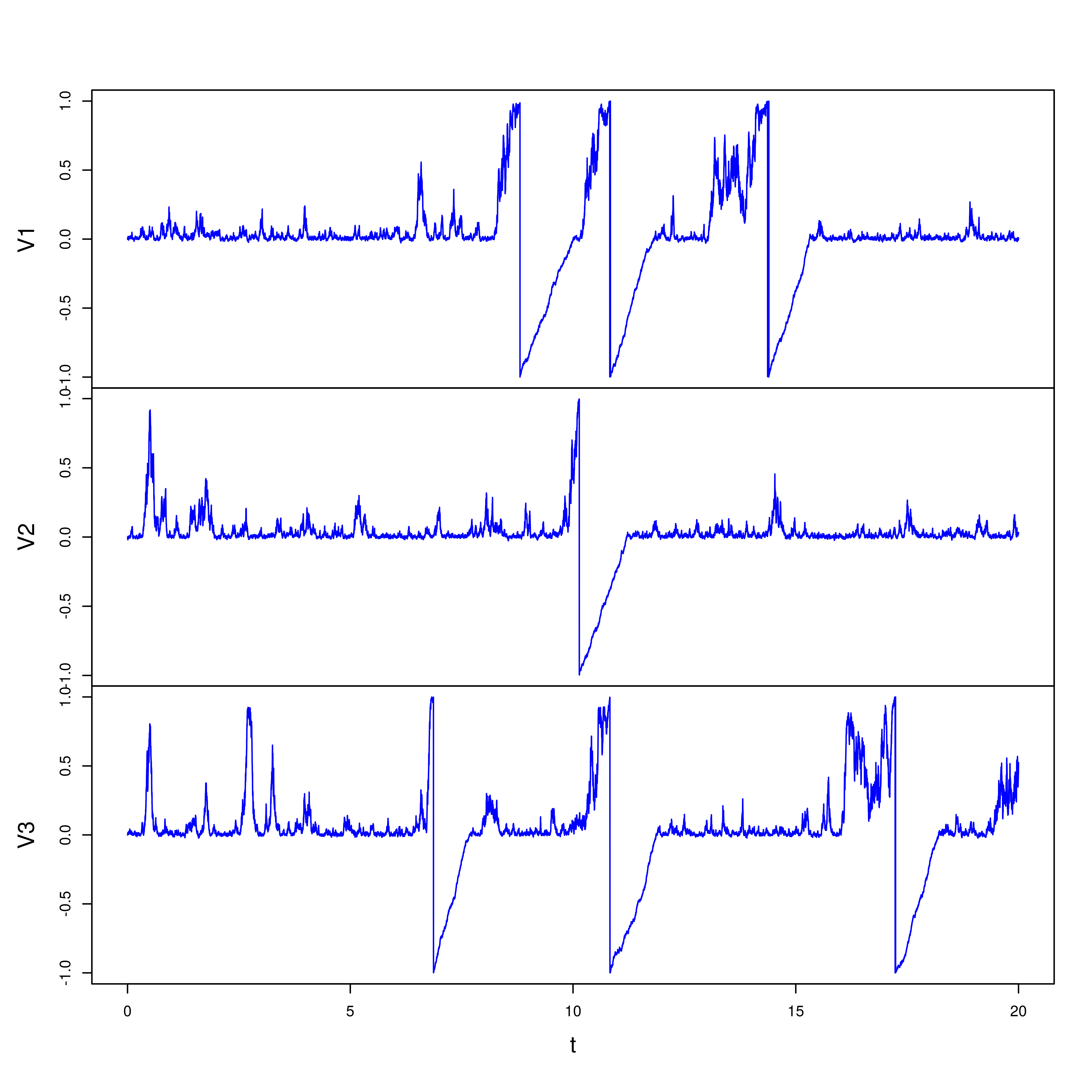}
  \caption{Solutions $V^{i,3}\pmod{2}$, $i=1,2,3$}
  \label{fig:1b}
\end{subfigure}
\caption{Simulation of the system (\ref{eq:dV}) for $3$ neurons with no spatial interaction.}

\label{fig:1}
\end{figure}
One could use as $\lambda$ any bounded function on $[0,2)$ that has a jump discontinuity in $v=1$ and is continuous elsewhere; all the results herein apply in this case with no modifications in the arguments, therefore we stick to the simple case given just above.
 \\
We will show that for each $N$ the system of equations (\ref{eq:dV}) has a unique solution $\left(V^{i,N}\right)_{i=1,\dots,N}$ with $V^{i,N}$ having continuous trajectories in $\bR$. This forces the trajectories of $V^{i,N}\pmod{2}$ to have jump discontinuities at every $t$ such that $V^{i,N}_t\in 2\bZ$. However continuity is easily restored seeing $V^{i,N}\pmod{2}$ as a process with values on the torus $\bT:=\bR/2\bZ$ (where $2\bZ$ is seen as a subgroup of translations on $\bR$). This corresponds to considering the interval $[0,2]$ with the identification $0\equiv 2$. Moreover $\bT$ is homeomorphic to $\nicefrac{1}{\pi}S^1\subset\bR^2$, the circle with radius $\nicefrac{1}{\pi}$.

On $\bT$ we consider the metric
\begin{equation}
  \label{eq:metric}
  d_\bT(v_1,v_2)=\min\left\{(v_1-v_2)\pmod{2},(v_2-v_1)\pmod{2}\right\}
\end{equation}
where on the right-had side $v_1$ and $v_2$ are seen as elements in $[0,2)\subset\bR$; this corresponds to the shortest-path (or geodesic) metric, which is the arc-length on $S^1$. This metric induces the quotient topology on $\bT$.\\
We will always consider the Euclidean metric on $D$ and endow $D\times\bT$ with the product metric, denoted by $d_{D\times\bT}$.\\
The choice to represent solutions on the torus is natural since the coefficients we introduced above are clearly $2$-periodic. To stress periodicity and also to lighten notation for $v\in\bR$ and $x,y\in D$ we define the functions
\begin{gather*}
  \lambda_2(v):=\lambda\big(v\pmod{2}\big)\ , \\
  g_2(x,v,y,w):=\theta(x,y)\ind_{[1,1+\delta]}\big(w\pmod{2}\big)\ind_{[0,1]}\big(v\pmod{2}\big)\ ,\\
  \sigma^\epsilon_2(v):=\sigma^\epsilon\big(v\pmod{2}\big)\ 
\end{gather*}
so that equation (\ref{eq:dV}) takes the more readable form
\begin{equation}
  \label{eq:dV2}
  \ud V^{i,N}_t=\lambda_2\left(V^{i,N}_t\right)\ud t+\frac{1}{N}\sum_{j=1}^{N}g_2\left({\xi^i},V^{i,N}_t,{\xi^j},V^{j,N}_t\right)\ud t+\sigma^\epsilon_2\left(V^{i,N}_t\right)\ud B^i_t\ .
\end{equation}
{Since $\bT$ is homeomorphic to $\nicefrac{1}{\pi}S^1\subset\bR^2$, one can define the Lebesgue measure on $\bT$ as the push-forward of the Lebesgue measure on $[0,2)$ through the map $t\mapsto(\cos (\pi t),\sin (\pi t))$; since $\bT$ is endowed with the quotient topology, any measure on the Borel sets of $\bT$ can be obtained in this way. Therefore we can interpret a Borel measure on $D\times\bT$ as a Borel measure on $D\times[0,2)$, and we will do so henceforth. Notice that any Borel measure on $\bT$ defines a Borel measure on the whole $\bR$ by $2$-periodic replication; we will not distinguish between the two in the sequel}.\\
{ We will show that the solution to (\ref{eq:dV2}) has a density which is $2$-periodic, thanks to the form of the coefficients; hence this density can be identified with a Borel measure on $D\times\bT$.}
\begin{figure}[!tbp]
  \centering
  \begin{subfigure}[b]{0.49\textwidth}
    \includegraphics[width=\textwidth]{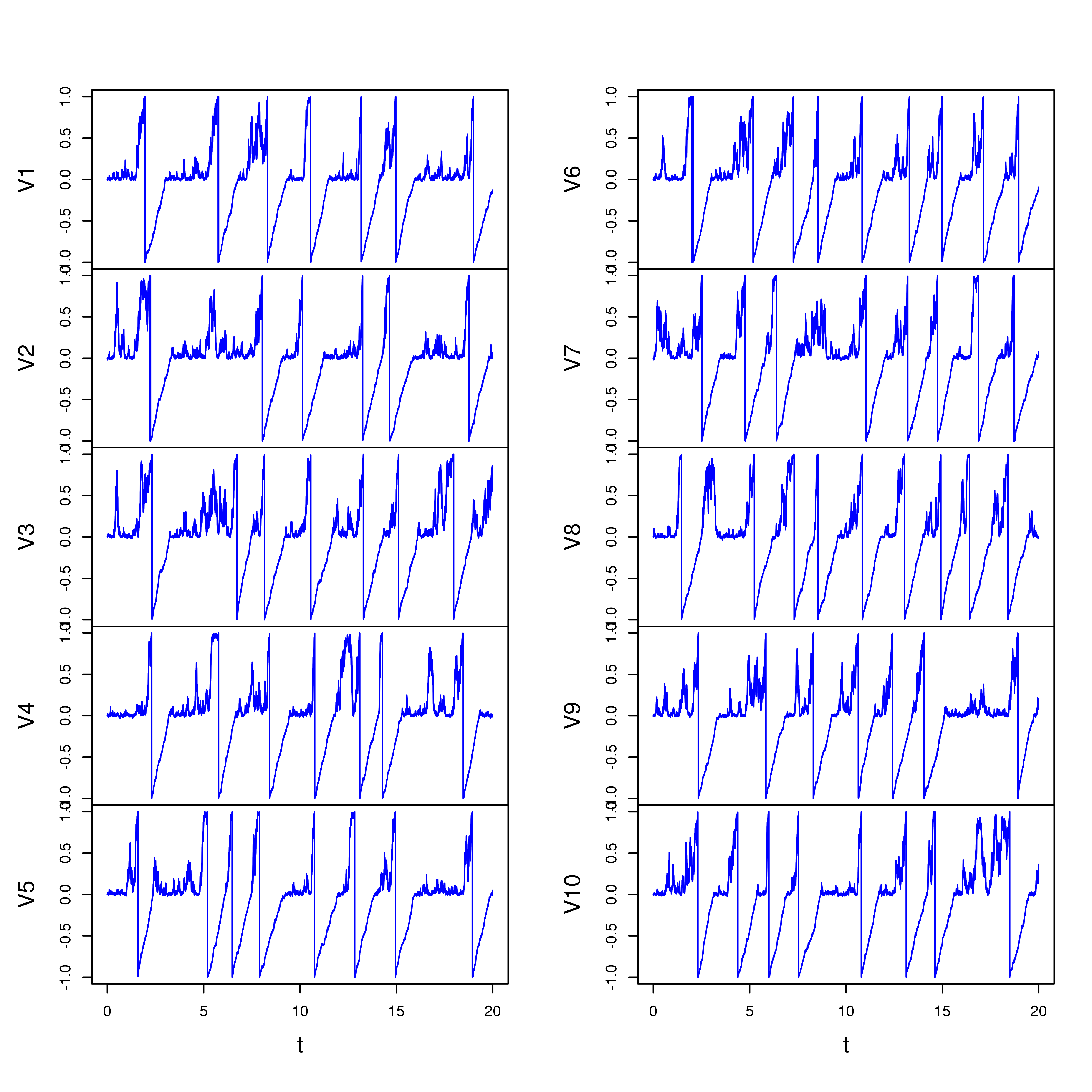}
\caption{Solutions $V^{i,10}$}
\label{fig:2a}
  \end{subfigure}
\hfill
  \begin{subfigure}[b]{0.49\textwidth}
    \includegraphics[width=\textwidth]{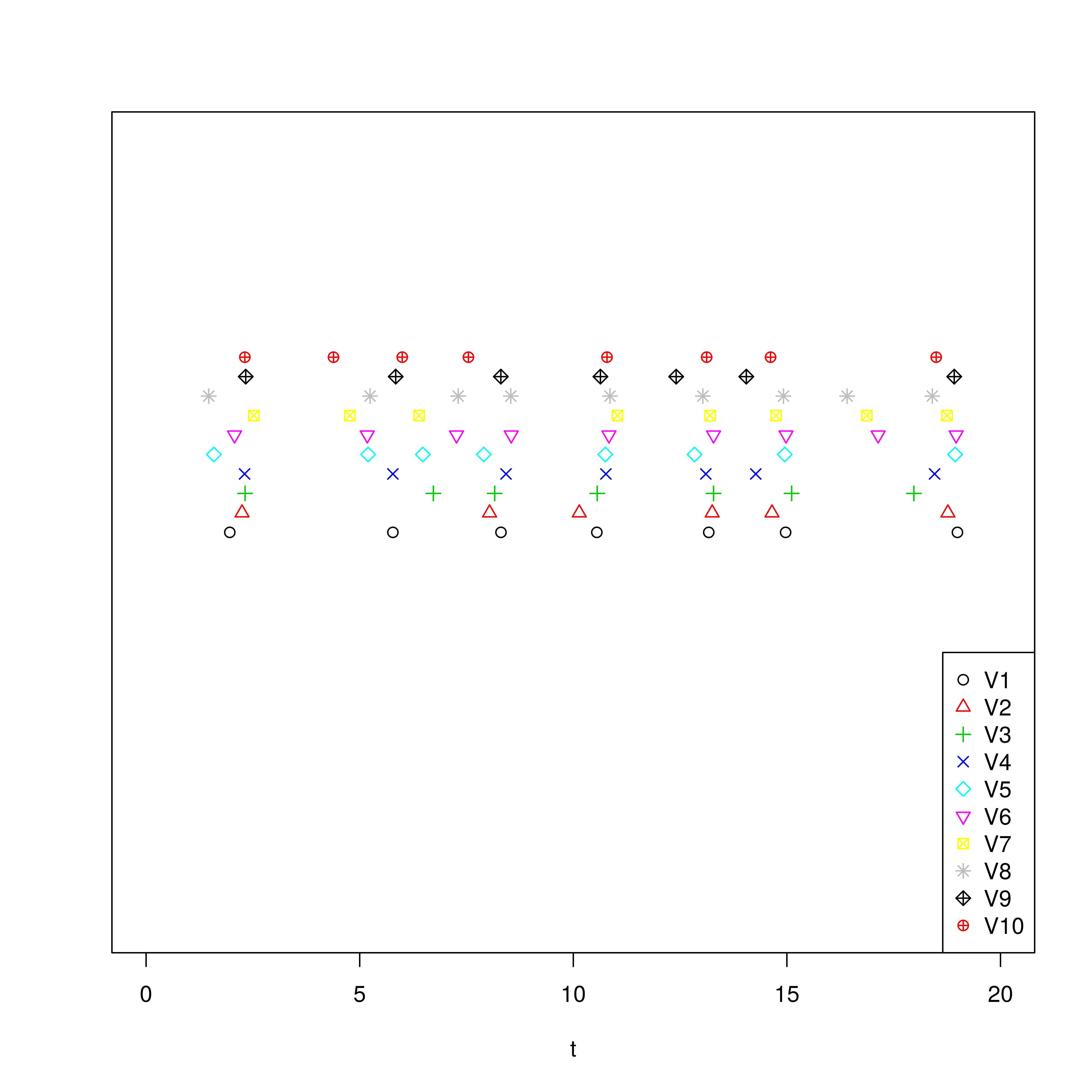}
\caption{Spikes of each neuron in time}
\label{fig:2b}
  \end{subfigure}
\label{fig:2}
\caption{Solution of system (\ref{eq:dV}) for $10$ neurons with strong interaction uniform in space. A spike from one neuron propagates to all other neurons in the network}
\end{figure}

Let $S_t^N$ denote the empirical measure
\begin{equation}
\label{eq:emp}
  S_t^N=\frac{1}{N}\sum_{i=1}^N\delta_{\left({\xi^i},V_t^{i,N}\pmod*{2}\right)}=\frac{1}{N}\sum_{i=1}^N\delta_{\left(\xi^i,V^{i,N}_t\pmod*{2}\right)}\ 
\end{equation}
and set for every $x,y\in D$ and $v\in\bR$
\begin{equation*}
  \sigma^\epsilon_2(x,v):=\sigma^\epsilon_2(v)\ ,\quad\theta(x,y,v):=\theta(x,y)\quad\text{ and }\quad \lambda_2(x,v):=\lambda_2(v)
\end{equation*}
for later use {(see for instance (\ref{eq:FP}))}. {To any function $\phi$ on $D\times\bT$ which is continuous corresponds a unique continuous function on $D\times\bR$ that is $2$-periodic with respect to its second variable, given by $(x,v)\mapsto\phi(x,v\pmod*{2})$; in the sequel we will identify these two functions and denote by $\phi$ also its $2$-periodic representation on $D\times\bR$.
With this convention, for any smooth and compactly supported function $\phi$ on $D\times\bT$ we have}
\begin{multline}
\label{eq:Ito_SN}
\ud\langle S^N_t,\phi\rangle=\ud\frac{1}{N}\sum_{i=1}^N\phi \left(X_0^i,V_t^{i,N}\right) \\
 =\frac{1}{N}\sum_{i=1}^N\partial_v\phi \left(X_0^i,V^{i,N}_t\right)\left(\lambda_2\left(V^{i,N}_t\right)+\frac{1}{N}\sum_{j=1}^Ng_2\left(X_0^i,V^{i,N}_t,X_0^j,V^{j,N}_t\right)\right)\ud t \\
 +\frac{1}{N}\sum_{i=1}^N\sigma^\epsilon_2\left(V^{i,N}_t\right)\partial_v\phi \left(X_0^i,V^{i,N}_t\right)\ud B^i_t+\frac{1}{N}\sum_{i=1}^N\frac{\sigma^\epsilon_2\left(V^{i,N}_t\right)^2}{2}\partial^2_v\phi \left(X_0^i,V^{i,N}_t\right)\ud t\\[3mm]
=\langle S^N_t,\lambda_2\partial_v\phi \rangle\ud t+\langle S^N_t,\langle S^N_t,g_2(x,v,\cdot,\cdot)\rangle\partial_v\phi \rangle\ud t +\langle S^N_t,\frac{\left(\sigma_2^\epsilon\right)^2}{2}\partial^2_v\phi \rangle\ud t+\ud M_t^{N,\phi}\ 
\end{multline}
where we use the notation
\begin{equation}
  \label{eq:notaz}
\langle S^N_t,\langle S^N_t,g_2(x,v,\cdot,\cdot)\rangle\partial_v\phi \rangle =  \int_{D\times\bT}\partial_v\phi (x,v) S^N_t(\ud x,\ud v)\int_{D\times\bT}g_2(x,v,y,w)S^N_t(\ud y,\ud w)
  \end{equation}
and
\begin{equation*}
  M^{N,\phi}_t=\int_0^t\frac{1}{N}\sum_{i=1}^N\sigma^\epsilon_2\left(V^{i,N}_s\right)\partial_v\phi \left(X^i,V^{i,N}_s\right)\ud B^i_s
\end{equation*}
is a martingale such that $\bE \big [ \sup_t\left\vert M^{N,\phi}_t\right\vert^2 \big]\to 0$ as $N\to\infty$ (due to the stochastic integrals being uncorrelated).\\
If we suppose that the sequence of random measures $S^N_t(\ud x,\ud v)$ converges in probability (in a suitable space) to a probability measure $\rho_t(x,v)\ud x\ud v$ on $D\times\bT$, then, heuristically, a passage to the limit in $N$ suggests that $\rho_t$ solves weakly the partial differential equation of Fokker-Planck type
\begin{multline}
\label{eq:FP}
  \frac{\partial}{\partial t}\rho_t(x,v)=\frac{1}{2}\frac{\partial^2}{\partial v^2}\left(\left(\sigma_2^\epsilon\right)^2\rho_t\right)(x,v)-\frac{\partial}{\partial v}\left(\lambda_2\rho_t\right)(x,v)\\-\frac{\partial}{\partial v}\left(\rho_t\int g_2(\cdot,\cdot,y,w)\rho_t(y,w)\ud y\ud w\right)(x,v)\ .
\end{multline}
In the sequel we will prove rigorously a similar assertion involving measures $\mu_t$ instead of densities $\rho_t$.
\begin{figure}[!]
  \centering
   \begin{subfigure}[b]{0.32\textwidth}
 \includegraphics[trim={8cm 0 8cm 0},clip,width=\textwidth]{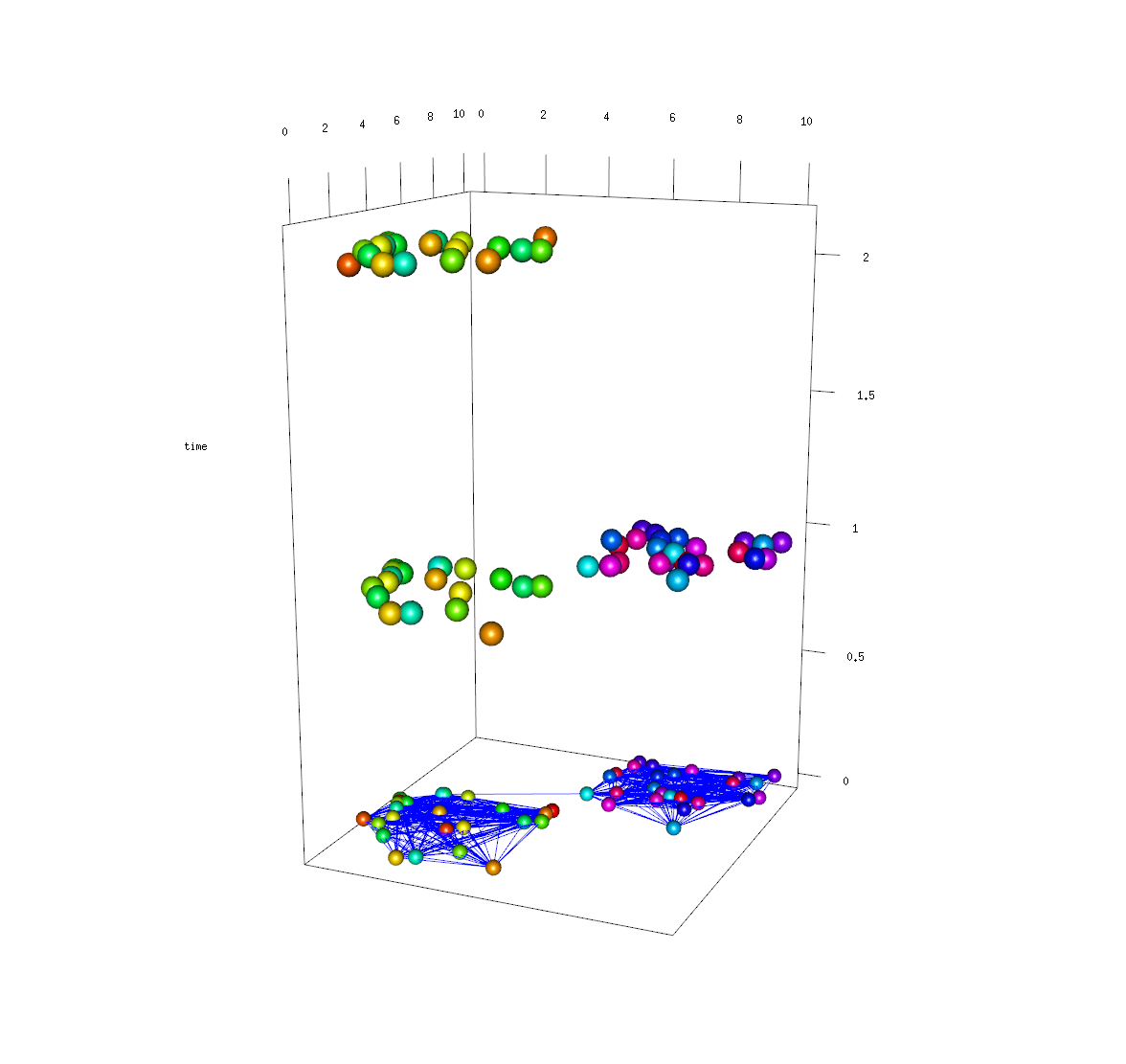}
 \label{fig:3a}    
   \end{subfigure}
\hfill
   \begin{subfigure}[b]{0.32\textwidth}
 \includegraphics[trim={8cm 0 8cm 0},clip,width=\textwidth]{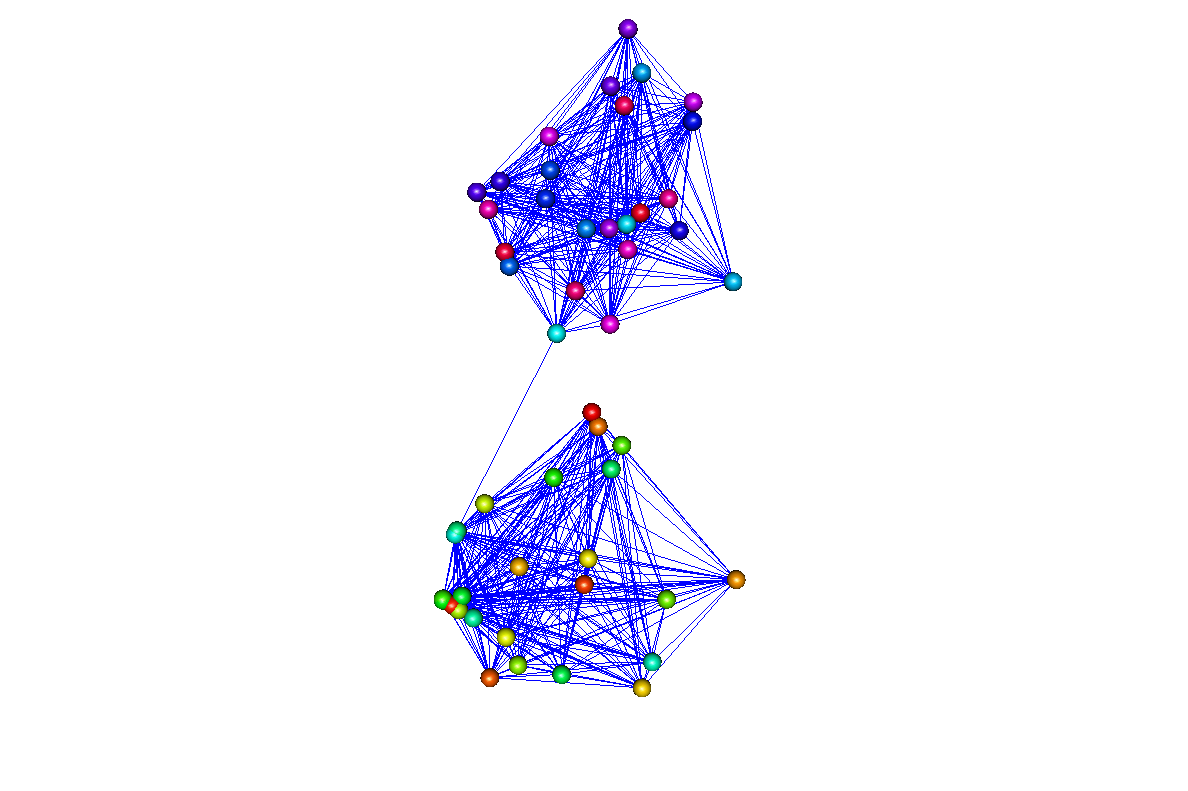}
 \label{fig:3ab}
 \end{subfigure}
\hfill
 \begin{subfigure}[b]{0.32\textwidth}
 \includegraphics[trim={9cm 0 8cm 0},clip,width=\textwidth]{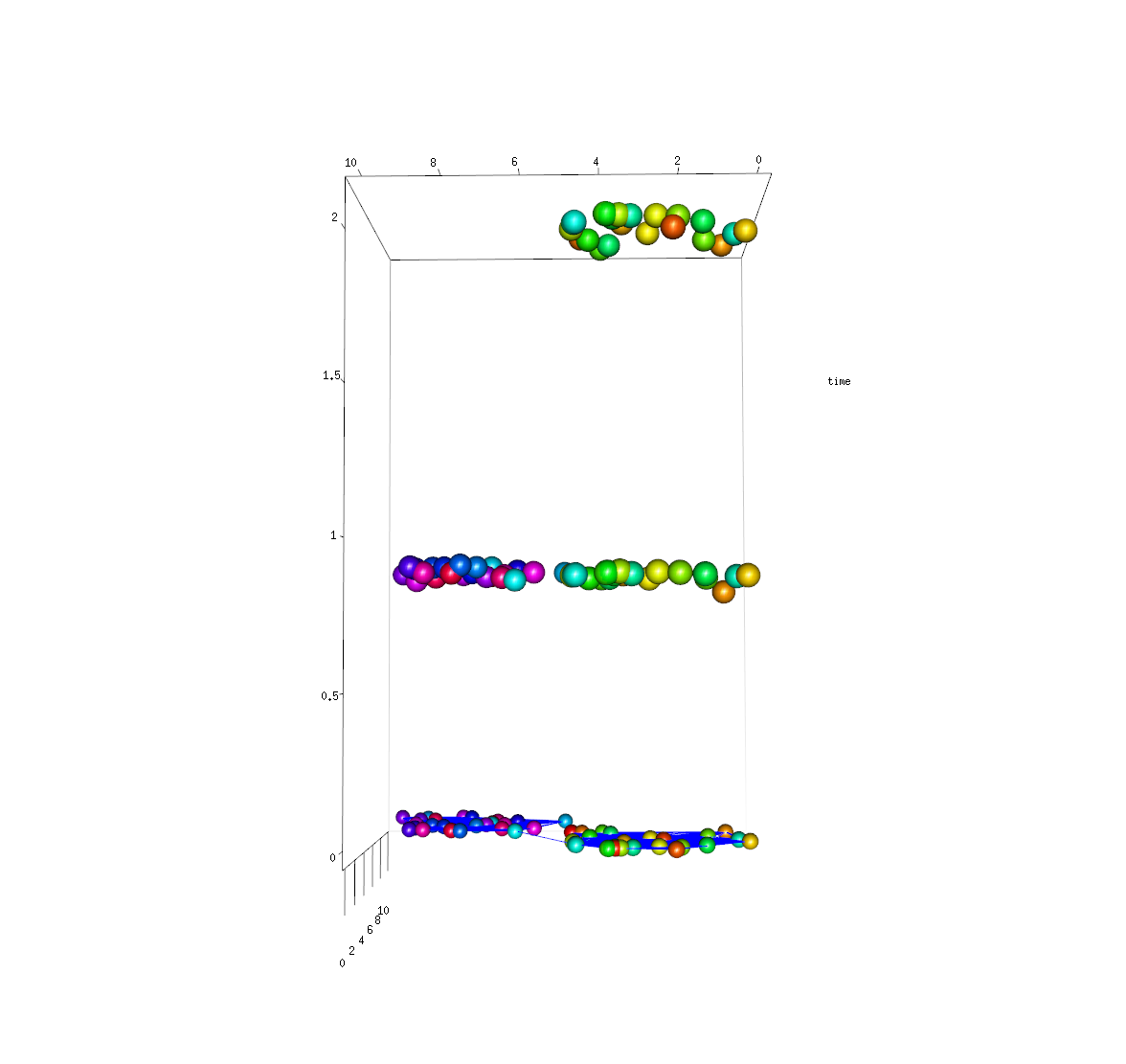}
 \label{fig:3c}    
   \end{subfigure}
\label{fig:3}
\caption{A network made of two subnetworks with localized strong interaction. Each blue line corresponds to the presence of interaction; only one neuron of the first subnetwork interacts with a single neuron of the second one. The vertical axis represents time; the network is drawn on the $t=0$ plane and to each spike of a neuron corresponds a sphere above it, with the same color as the neuron. Near time $t=1$ many neurons in the first network spike and the signal is propagated to the second network, while at time $t=2$ the signal does not propagate to the second network.}
\end{figure}
\subsection{Main results}
The main aim of the paper is to show that the empirical measure actually converges in a weak sense to a limiting probability measure $\mu_t(\ud x,\ud v)$ such that the marginal with respect to $v$ has a $L^2$-density and which is the unique solution to the above PDE (\ref{eq:FP}). This is the content of Theorem \ref{thm:FP}.\\
The paper is organized as follows. In Section \ref{section:SDE} we will prove strong well-posedness for the system of SDEs (\ref{eq:dV}); a modification of the standard theory for finite-dimensional SDEs with bounded and measurable drift is needed here to deal with the dependence of the equations on the random variables $X_0^i$. In Section \ref{section:uniqueness} we define a weak measure-valued solution and show that the PDE (\ref{eq:FP}) has at most one such solution.
\\
To show existence of a solution to the PDE we first prove that the laws $\bQ^N$ of the empirical measures of $(X^i,V^{i,N}\pmod*{2})$ (see (\ref{eq:emp})) are tight as probability measures on the space of continuous measure-valued functions of time (Section \ref{section:Q_tight}). Then we prove that any limit point $\bQ$ of $\bQ^N$ gives full measure to the set of functions with values that are continuous measures with marginal with respect to $v$ having a $L^2$-density, that $\bQ$ is supported by the set of weak measure-valued solutions to the Fokker-Planck PDE and that actually all the sequence $\bQ^N$ converges to the same limit. This provides existence of a solution and is discussed in Section \ref{section:PDE_existence}.\\
Section \ref{section:vlasov} briefly shows that well-posedness of the Fokker-Planck PDE implies existence of a unique strong solution to the McKean-Vlasov SDE associated with the particle system. We conclude with an appendix giving some immediate generalizations of our results together with some indications on how the proofs have to be adapted to this more general setting.
\section{The system of particles}
\label{section:SDE}
Consider independent Brownian motions $B^i$, $i\in\bN$ and assume that the random variables $\eta^i$ introduced in the previous section are i.i.d. and independent from all $B^i$; denote by $(\sF^0_t)_t$ the filtration $\sF^0_t=\sigma\left(B^i_s,0\leq s\leq t,i\in\bN\right)\vee\sigma(\eta^i,i\in\bN)$ augmented with the $\bP$-null sets.\\
We also denote by ${\cal G}$ the $\sigma$-algebra $\sigma\left(\xi^i,i\in\bN\right)$ and assume that for any $t \ge 0$, ${\mathcal G}$ and ${\mathcal F}_t^0$ are independent.\\
Finally we introduce the filtration $(\sF_t)$  where $\sF_t$ is the completion of $\sF_t^0\vee\sG$.\\

Let us write our system of equations in vector form: we fix $N\in\bN$ and introduce, for the variables $\mathbf v=(v^{1},\dots,v^{N})^\top\in\bR^N$, $\mathbf x=(x^{1},\dots,x^{N})^\top\in D ^N$, the functions 
\begin{gather*}
  \bar\lambda_2:\bR^N\to\bR^N\\
  \bar\lambda_2(\mathbf{v})=\left(\lambda_2(v^{1}),\dots,\lambda_2(v^{N})\right)^\top\ ,
\end{gather*}
\begin{gather*}
  \bar g_2:\left(D\times\bR\right)^{N+1}\to\bR^N\\
  \bar g_2(\mathbf x,\mathbf v,y,w)=\left(g_2(x^{1},v^{1},y,w),\dots,g_2(x^{N},v^{N},y,w)\right)^\top,
\end{gather*}
\begin{gather*}
  \bar\sigma_2^\epsilon:\bR^N\to\bR^{N\times N}\\
  \bar\sigma_2^\epsilon(\mathbf v)=\diag\left(\sigma_2^\epsilon(v^{1}),\dots,\sigma_2^\epsilon(v^{N})\right).
\end{gather*}
Setting $\Xi=\left(\xi^1,\dots,\xi^N\right)$, $\Psi=\left(\eta^1,\dots,\eta^N\right)$, we want to show existence and uniqueness of strong solutions to
\begin{equation}
  \label{eq:system}
  \begin{cases}
    \ud \mathbf V_t=\bar\lambda_2\left( \mathbf V_t\right)\ud t+\frac{1}{N}\sum_{j=1}^{N}\bar g_2\left(\Xi,\mathbf V_t,\Xi^j,V^{j}_t\right)\ud t+\bar\sigma^\epsilon_2\left(\mathbf V_t\right)\ud \mathbf B_t\\
    \mathbf V_0=\Psi\ .
  \end{cases}
\end{equation}
The classical reference for existence of a strong solution for SDEs with bounded measurable drift is \cite{V81}. However, the results proved therein do not apply directly to equation (\ref{eq:system}), { because they do not guarantee the measurable dependence of the solution $V$ on the stochastic parameter $\Xi$.}  Therefore we introduce the following definition.\\

Let us denote by $\sS$ the Banach space of all continuous paths from $[0,T]$ into $\bR^N$ endowed with the supremum norm $\lvert\cdot \rvert_\sS$. We also denote by $\nu_N$ the law of the random variable $\Xi$ on the Borel $\sigma$-algebra of $D^N$.

\begin{definition}[Strong solution]
\label{def:strongsolution}
 A \emph{strong solution} to (\ref{eq:system}) is a family of continuous $\bR^N$-valued ${\cal F}_t^0$-adapted processes $({\bf V}^{{\bf x}}_t)$, ${\bf x} \in D^N$, such that
\begin{enumerate}[label=$($\roman{*}$)$]
\item the mapping $({\bf x},\omega) \mapsto {\bf V}_{\cdot}^{{\bf x}}(\omega) \in \sS$ is measurable from $(D^N \times \Omega, {\cal B}(D^N) \times {\cal F}) $ to $(\sS, {\cal B}(\sS))$;

\item for $\nu_N$-almost every ${\bf x} \in D^N$, $({\bf V}_t^{{\bf x}})$ satisfies equation (\ref{eq:system}) in the strong sense when $\Xi={\bf x}$.
\end{enumerate}
 \end{definition}
The above definition is motivated by the fact that if $({\bf V}^{{\bf x}}_t)$ is a strong solution, then the ${\cal F}_t$-adapted process $({\bf V_t^{\Xi}})$ satisfies equation (\ref{eq:system}) $\bP$-almost surely for any $t\in[0,T]$. In fact, for $\Xi$ random variable as above,  the process $({\bf V}^{\Xi}_t)$ is well-defined with values in $\bR^N$, has  continuous paths and is ${\mathcal F}_t$-adapted.\\
To prove that $\left(\mathbf {V}^{\Xi}_t\right)$ satisfies (\ref{eq:system}) it is enough to compute, using conditional expectation with respect to $\cal G$ and independence,
\begin{align*}
&\bE  \left \vert \Psi + \int_0^t \bar\lambda_2\left( \mathbf { V}^{\Xi}_r\right)\ud r +
\frac{1}{N}\sum_{j=1}^{N} \int_0^t \bar g_2\left(\Xi,\mathbf {V}^{\Xi}_r,\Xi^j,{V}^{\Xi,j}_r\right)\ud r+ \int_0^t \bar\sigma^\epsilon_2\left(\mathbf {V}^{\Xi}_r\right)
\ud \mathbf B_r - {\bf V}_t^{\Xi} \right\vert\\
&=\bE \bE \left [ \left \vert \Psi + \int_0^t \bar\lambda_2\left(\mathbf{V}^{\Xi}_r\right)\ud r+\frac{1}{N}\sum_{j=1}^{N} \int_0^t \bar g_2\left(\Xi,\mathbf {V}^{\Xi}_r,\Xi^j,V^{\Xi,j}_r\right)\ud r+ \int_0^t \bar\sigma^\epsilon_2\left(\mathbf {V}^{\Xi}_r\right)\ud \mathbf B_r - {\bf V}_t^{\Xi} \right\vert \ {\cal G}\right]\\
&=\bE  \bE \left [ \left \vert \Psi + \int_0^t \bar\lambda_2\left( \mathbf {V}^{\mathbf x}_r\right)\ud r+\frac{1}{N}\sum_{j=1}^{N} \int_0^t \bar g_2\left(\mathbf x,\mathbf {V}^{\mathbf x}_r,x^j,{V}^{{\mathbf x},j}_r\right)\ud r+ \int_0^t \bar\sigma^\epsilon_2\left(\mathbf {V}^{\mathbf x}_r\right)\ud \mathbf B_r - {\bf V}_t^{\mathbf x}\right\vert_{{\bf x} = \Xi}\right]\\
&=0.
\end{align*}

\begin{theorem}
\label{thm:SDEE1}
  For every $N\in\bN$ there exists a strong solution to  (\ref{eq:dV2}). Two strong solutions on the same probability space associated to the same initial condition $\Psi$ are indistinguishable for $\nu_N$-almost every ${\bf x}\in D^N$.
\end{theorem}
\begin{proof}
{\it Existence.} First fix 
${\bf x} \in D^N$; the SDE
\begin{equation}
  \label{eq:system3}
  \begin{cases}
    \ud \mathbf V^{\mathbf x}_t=\bar\lambda_2\left(\mathbf V^{\mathbf x}_t\right)\ud t+\frac{1}{N}\sum_{j=1}^{N}\bar g_2\left(\mathbf x,\mathbf V^{\mathbf x}_t,x^j,V^{{\mathbf x},j}_t\right)\ud t+\bar\sigma^\epsilon_2\left(\mathbf V^{\mathbf x}_t\right)\ud \mathbf B_t\\
    \mathbf V^{\mathbf x}_0=\Psi\ .
  \end{cases}
\end{equation}
admits a unique strong solution $\mathbf V^{\mathbf x}$ by the results proved in \cite{V81}. We now clarify the measurability of ${\bf V^x}$ with respect to ${\bf x}$. One difficulty is that the proof of the main well-posedness results in \cite{V81} is based on the Yamada-Watanabe theorem and is indeed  abstract and non-contructive. This is why to prove such a measurability property we follow the approach in \cite{GK96}. Fix a sequence of partitions $\left\{\pi^n\right\}_{n\in\bN}$ of $[0,T]$, where each $\pi^n$ is given by points $0=t_0^n<t_1^n<\dots<t_n^n=T$, set $\kappa^n(t)=\sum_{i=0}^{n-1}t_i^n\ind_{[t_i^n,t_{i+1}^n)}(t)$ and consider Euler's approximations to equation (\ref{eq:system3}) given by
\begin{equation}
  \label{eq:system_approx}
  \begin{cases}
    \ud \mathbf V^{\mathbf x,n}_t=\bar\lambda_2\left( \mathbf V^{\mathbf x,n}(\kappa^n(t))\right)\ud t+\frac{1}{N}\sum_{j=1}^{N}\bar g_2\left(\mathbf x,\mathbf V^{\mathbf x,n}(\kappa^n(t)),x^j,V^{{\mathbf x},n,j}(\kappa^n(t))\right)\ud t\\\phantom{aaaaaaaaaaaaaaaaaaaaaaaaaaaaaaaaaaaaaaaaaaaaaa}+\bar\sigma^\epsilon_2\left(\mathbf V^{\mathbf x,n}(\kappa^n(t))\right)\ud \mathbf B_t,\\
    \mathbf V^{\mathbf x,n}_0=\Psi\ .
  \end{cases}
\end{equation}
In the next steps,  we will  use that, for any $n \ge 1$, the  processes ${\bf V^{x,n}}$  enjoy all the measurability properties we need.\\
By Theorem 2.8 in \cite{GK96}
we know that, for any ${\bf x} \in D^N$ and $\delta>0$
 \begin{equation}
\label{c11}
 \bP \left(\left\vert{\bf V}^{{\bf x},n}- {\bf V}^{\bf x}\right\vert_\sS > \delta\right) \to 0, \;\; \text{as }n \to \infty. 
 \end{equation} 
Note that, for any $p\geq 2$ (using the boundedness of the coefficients of the SDEs) there exists $C_p>0$ (independent of $n$ and $\bf{x}\in D$) such that, for any $n \ge 1$, $\bE \left[\left\vert {\bf V}^{{\bf x},n}- {\bf V}^{\bf x}\right\vert^{p}_\sS\right]$ $\le C_p$. Writing 
\begin{multline*}
\bE \left\vert{\bf V}^{{\bf x},n}- {\bf V}^{\bf x}\right\vert^2_\sS \\ = \bE \left[\left\vert{\bf V}^{{\bf x},n}- {\bf V}^{\bf x}\right\vert^2_\sS \, \ind_{ \{  \vert{\bf V}^{{\bf x},n}- {\bf V}^{\bf x}\vert_\sS > \delta \} }\right]+ \bE \left[\left\vert{\bf V}^{{\bf x},n}- {\bf V}^{\bf x}\right\vert^2_\sS \, \ind_{ \{  \vert{\bf V}^{{\bf x},n}- {\bf V}^{\bf x}\vert_\sS \le  \delta \} }\right]
\end{multline*}
and using also the H\"older inequality, we easily deduce that, for any ${\bf x} \in D^N$, 
 \begin{equation} \label{c112}
\bE\left[\left\vert{\bf V}^{{\bf x},n}- {\bf V}^{\bf x}\right\vert_\sS^2 \right] \to 0, \;\; \text{as }n \to \infty\ . 
\end{equation} 
Hence, for any ${\bf x} \in D^N$, ${\bf V}^{{\bf x},n}$ converges to ${\bf V^{x}}$ in the Banach space $L^2(\Omega; \sS)$ and so the mapping ${\bf x} \mapsto {\bf V}^{{\bf x}} \in L^2(\Omega; \sS)$ is Borel measurable on $\left(D^N,\sB(D^N)\right)$. By the dominated convergence theorem we infer
\begin{equation} \label{de1}
\int_{D^N} \bE\left[\left\vert{\bf V}^{{\bf x},n}- {\bf V}^{\bf x}\right\vert_\sS^2 \right]\nu_N(\ud x) \to 0, \;\; \text{as }n \to \infty\ ,
\end{equation} 
i.e., $({\bf V}^{{\bf \cdot} ,n})$ converges to ${\bf V^{\cdot}}$ in $\sZ =L^2 \left((D^N,\nu_N); L^2(\Omega; \sS)\right)$. It follows that $({\bf V}^{{\bf x},n})$ is a Cauchy sequence in $\sZ$. Using that, for any $n, m \ge 1,$ 
\begin{equation} \label{w2}
\int_{D^N} \bE\left[\left\vert{\bf V}^{{\bf x},n}- {\bf V}^{{\bf x}, m}\right\vert_\sS^2 \right]\nu_N(\ud x) = \int_{D^{N} \times \Omega}  \left\vert{\bf V}^{{\bf x},n}(\omega)- {\bf V}^{{\bf x}, m}(\omega)\right\vert_\sS^2 \nu_N(\ud x) \bP(\ud \omega)\ ,
\end{equation} 
we get that $({\bf V}^{{\bf \cdot} ,n})$ converges to some ${\bf \tilde V^{\cdot}}$ in  $L^2 \left((D^N \times \Omega,\nu_N\times\bP); \sS\right)$. In particular, ${\bf \tilde V^{\cdot}}$ is measurable on $(D^N \times \Omega, {\mathcal B}(D^N) \times {\mathcal F})$ with values in $\sS$.\\
It follows that, for a.e. ${\bf x} \in D^N$, we have ${\bf \tilde V^{x}} = {\bf V^{x }}$ in $\sS$, $\bP$-a.s. (we have obtained a version of the strong solution which has the required measurability properties with respect to ${\bf x}$).

\smallskip
\noindent {\it Uniqueness.}  It follows directly from the celebrated Veretennikov result.
\end{proof}

\begin{remark}
\label{rem:sigma}
To show existence and uniqueness of a solution one could weaken the assumptions on the regularity of $\sigma^\epsilon$, similarly to what is done in the references \cite{V81} and \cite{GK96}. What is needed above is that pathwise uniqueness holds for equation (\ref{eq:system3}), and there are many well-known conditions assuring that this happens. However at a later stage in the paper (Section \ref{section:PDE_existence}) we will need to assume that $\sigma^\epsilon$ is differentiable with bounded derivative. This does not seem to be a limitation on the model, since there is no reason to assume that the diffusion coefficient be particularly rough.
\end{remark}
\section{The limit PDE: uniqueness of measure-valued solutions}
\label{section:uniqueness}
Let $\mathrm{Pr}(D\times\bT)$ be the space of Borel probability measures over $D\times\bT$. Let $\mathrm{Pr}_1(D\times\bT)\subset\mathrm{Pr}(D\times\bT)$ be the space of probability measures over $D\times\bT$ with finite first moment, endowed with the 1-Wasserstein metric $\sW_1$.\\
Suppose as above that the empirical measures $S^N_t$ converge in a weak sense to a probability measure $\mu_t$ on $D\times\bT$. Without assuming that $\mu_t$ has a density, we expect that it solve the PDE

\begin{equation}
\label{eq:FP2}
  \frac{\partial}{\partial t}\mu_t=\frac{1}{2}\frac{\partial^2}{\partial v^2}\left(\left(\sigma_2^\epsilon\right)^2\mu_t\right)-\frac{\partial}{\partial v}\left(\lambda_2\mu_t\right)-\frac{\partial}{\partial v}\left(\mu_t\int g_2(\cdot,\cdot,y,w)\mu_t(\ud y,\ud w)\right)\ ,
\end{equation}
with initial condition $\nu\times\tilde\rho_0\rL_{\bT}\ $, meaning that $\mu_0(\ud x,\ud v)=\nu(\ud x)\tilde\rho_0(v)\rL_\bT(\ud v)$.
Fix $T>0$; we will denote by $\rC$ the space
\begin{equation*}
\rC:=C\left([0,T];\mathrm{Pr}_1(D\times\bT)\right)
\end{equation*}
and for a measure $\zeta\in\mathrm{Pr}(D\times\bT)$ we will adopt the notation
\begin{equation}
\label{eq:notation_b}
b(\zeta)(x,v):=\lambda_2(v)+\int_{D\times\bT}g_2(x,v,y,w)\zeta(\ud y,\ud w),\;\; x\in D,\, v\in\bT,
\end{equation}
throughout the rest of the paper.\\
In the sequel we will often use  that  if $g: \bR \to \bR$ is 2-periodic and differentiable on $\bR$ then its derivative is also $2$-periodic (thus $g$ can be identified with a differentiable function on $\bT$). Recall the Banach space $B_b(D\times\bT)$ consisting of all Borel and bounded functions $f: D \times \bT \to \bR$ endowed with the supremum norm $\| \cdot \|_{\infty}$. We will also consider $C_b(D \times \bT)\subset B_b(D\times\bT)$ consisting of all bounded continuous functions. We introduce the space of test functions
  \begin{equation*}
    \sT\subset C_b( D \times \bT)
  \end{equation*}
defined as the space of all $\phi \in C_b( D \times \bT)$ such that there exist the partial derivatives $\partial_v\phi$ and $\partial_{v}^2 \phi$ which both belong to $C_b(D \times \bT)$.
\begin{definition}
We say that $\mu\in \rC$ is a weak measure-valued solution of the nonlinear Fokker-Planck equation (\ref{eq:FP2}), with initial condition $\mu^{0}\in \mathrm{Pr}_{1}\left(D\times\bT\right)$, if
\begin{equation}
\label{FP weak}
\left\langle \mu_{t},\phi\right\rangle =\left\langle \mu^{0},\phi\right\rangle+\int_{0}^{t}\left\langle \mu_{s},b(\mu_s)  \partial_{v}\phi\right\rangle \ud s+\int_{0}^{t}\left\langle \mu_{s},\frac{\left(\sigma_{2}^{\epsilon}\right)^{2}}{2}\partial_{v}^{2}\phi\right\rangle\ud s, 
\end{equation}
for every test function $\phi \in {\cal T}$.
\end{definition}

Consider now the total variation distance on Borel probability measures on $D\times\bT$
\begin{align*}
  d_{\mathrm{TV}}(\nu^1,\nu^2):&=\sup\left\{\left\vert\langle\nu^1,\phi\rangle-\langle\nu^2,\phi\rangle\right\vert\colon\phi\in B_b(D\times\bT),\VV\phi\VV_\infty\leq 1\right\}\\
                 &=\sup\left\{\left\vert\langle\nu^1,\phi\rangle-\langle\nu^2,\phi\rangle\right\vert\colon\phi\in C_b(D\times\bT),\VV\phi\VV_\infty\leq 1\right\}\ .
\end{align*}
\begin{remark} \label{maa} Let $\mu^1,\mu^2\in \rC$. One can show that the mapping 
$$
t \mapsto  d_{\mathrm{TV}}(\mu^1_t,\mu^2_t)   \;\; \text{ is  Borel and bounded on $[0,T]$.}
$$
To this purpose we first remark that, for given probability measures $\nu^1$ and $\nu^2$ on Borel sets of $D \times \bT$, one has
 \begin{equation} \label{ci1}
 d_{\mathrm{TV}}(\nu^1,\nu^2)=\sup\left\{\left\vert\langle\nu^1,\phi\rangle-\langle\nu^2,\phi\rangle\right\vert\colon\phi\in C_c^{\infty}(D\times\bT),\VV\phi\VV_\infty\leq 1\right\}\ .
\end{equation}
(recall that $\phi \in C_c^{\infty}(D\times\bT)$ if $\phi \in C^{\infty}(D\times\bT) $ and has compact support).
 As before  if $f: D \times \bT \to \bR $ we identify the function $(x,v)\mapsto f(x,v\pmod*{2}) \in \bR$ defined on $D \times \bR$ with $f$.

To prove \eqref{ci1} let  $f \in C_b (D\times\bT)$ with $\VV f\VV_\infty\leq 1$, by truncating $f$  and by considering standard mollifiers (defined on $\bR^4$) we can find a sequence $(f_n) \subset C_c^{\infty}(D\times\bT)$ such that 
$\| f_n\|_{\infty} \le 1$ and $f_n(z) \to f(z)$, as $n \to \infty$, for any $z = (x, v) \in D \times \bT$. By using 
$$
\left\vert\langle\nu^1, f_n\rangle-\langle\nu^2,f_n \rangle\right\vert
\le
\sup\left\{\left\vert\langle\nu^1,\phi\rangle-\langle\nu^2,\phi\rangle\right\vert\colon\phi\in C_c^{\infty},\VV\phi\VV_\infty\leq 1\right\},\;\; n \ge 1,
$$
and the Lebesgue convergence theorem we obtain that the previous formula holds even when $f_n$ is replaced by $f$;  this leads to \eqref{ci1}.

Then we show that there exists a countable
set $K_\infty \subset C_c^{\infty}$ such
 that for any $f \in C_c^{\infty} $,
we can find  a sequence $(f_k) \subset K_\infty$ satisfying
 \begin{equation} \label{den1}
 \lim_{k \to \infty} \| f - f_k \|_{\infty} =0.
\end{equation}
As a simple consequence we get that  
\begin{equation} \label{den}
 d_{\mathrm{TV}}(\nu^1,\nu^2)=\sup\left\{\left\vert\langle\nu^1,\phi\rangle-\langle\nu^2,\phi\rangle\right\vert\colon\phi\in K_\infty,\VV\phi\VV_\infty\leq 1\right\}\ .
 \end{equation} 
  To prove  assertion \eqref{den1} set
$F_n  = \{ f \in C_c^{\infty}$ with support$(f)$ $ \subset B_n\}$ where   $B_n = \{ (x,v)  \in D \times \bT \, : \;   \dist(x, \partial D) \ge 1/n \;\ \text{ and }|x| \le n \}$.  Each $F_n $ is separable: indeed $F_n \subset C(B_n)$ and so there exists a countable set $K_n \subset F_n$ which is dense in $F_n$. 
 To finish  we define $K_\infty = \cup_{n \ge 1} K_n$.  \qed 
\end{remark}
\begin{theorem}
\label{thm:xd}
 Let $\mu^0\in\mathrm{Pr}_1(D\times\bT)$. There exists at most a unique weak measure-valued solution to equation (\ref{eq:FP2}), with initial condition $\mu^0$, in $\rC$.
\end{theorem}

\begin{proof}
Given a function $f : \bT \to \bR$ we still denote by $f$  its $2$-periodic version defined on $\bR$. Let $\psi \in \sT$ and define the operator 
$$A \psi (x,v)= A\psi(x, \cdot)(v) = \frac{\sigma^\epsilon_2(v)}{2}\frac{\partial^2 \psi}{\partial v^2}(x,v)\ .$$
It is well known that $A$ is the infinitesimal generator of a diffusion semigroup $T_t : B_b(D \times \bR) \to B_b(D \times \bR)$:
\begin{equation} \label{ciao}
 T_t \zeta(x,v) = \int_{\bR} \zeta(x,v') p_t(v,v') dv' , \; \; x \in D, \; v \in \bR, \; \zeta \in B_b(D \times \bR),\; t\ge 0,
\end{equation}
where the density $p_t(\cdot, v') \in C^2(\bR)$, for any  $v' \in \bR$,   $t>0$ (see, for instance, Chapter 6 in  \cite{Fri75}). Moreover, for $t>0$, $p_t(v,v')$, $\partial_v p_t (v,v')$ and $\partial^2_v p_t (v,v')$ are continuous functions on $\bR^2$.
 In addition, for any $g \in C_b(D \times \bR )$, $t \in (0,T)$, we have 
\begin{equation} \label{reg1}
 \| \partial_v T_t g \|_{\infty} 
 \le \frac{c}{\sqrt{t}} \| g\|_{\infty}, \;\;\;
  \| \partial_v^2 T_t g \|_{\infty} 
 \le \frac{c}{{t}} \| g\|_{\infty}.
\end{equation} 
Finally, if $f\in C_b^2 (D\times\bR)$, we have  $T_t f \in C^2_b(D \times \bR) $, $t\ge 0$, and $\partial_t T_t f(x,v) = T_t A f(x,v)$ $= AT_tf(x,v)$, $t\ge 0$, $x \in D$, $v \in \bR$.

Since in our case $\sigma_2$ is also $2$-periodic, it is not difficult to prove that, for $t>0,$  $p_t$ is  $2$-periodic in both variables, i.e., $p_t(v+ 2, v'+2) = p_t(v,v')$, $v,v' \in \bR$.
It follows that if $\psi \in B_b(D \times \bR)$ is $2$-periodic 
 in the $v$-variable 
then also $T_t \psi$ is $2$-periodic in the $v$-variable. Differentiating, we obtain that $\partial_v p_t (v,v')$ and $\partial^2_v p_t (v,v')$ are $2$-periodic continuous functions in both variables. Hence, in particular, $T_t \psi \in {\cal T}$ if $\psi \in {\cal T}$, $t \ge 0$.
\smallskip One can prove that $\mu\in \rC$ is a weak solution to (\ref{eq:FP2}) if and only if it is a mild solution, i.e., if and only if
\begin{equation}
\label{eq:FPmild}
  \langle \mu_t,\phi\rangle=\langle \mu^0, T_t \phi\rangle +\int_0^t\langle\mu_s,b(\mu_s)\frac{\partial}{\partial v} T_{t-s} \phi\rangle \ud s, \;\; \phi \in {\cal T}, t\in[0,T].
\end{equation}
We only show that any weak solution is a mild solution (this is the part  we need to prove our uniqueness claim).
 We fix $\phi $ and $t>0$. Differentiating with respect to $s \in (0,t)$ the mapping
$$
 s \mapsto  \langle T_{t-s} \phi , \mu_s \rangle =
\int_{D \times \bR} \mu_s(dx, dv) \int_{\bR} \phi (x, v') p_{t-s} (v,v') dv'
$$
we get 
\begin{equation*}
 \frac{d}{ds} ( \langle T_{t-s} \phi , \mu_s \rangle)
 = - \langle T_{t-s} A \phi , \mu_s \rangle) + \langle \mu_{s},b(\mu_s)  \partial_{v} T_{t-s} \phi\rangle  
 +
 \langle \mu_{s}, T_{t-s} A \phi \rangle\ .
 \end{equation*}
Integrating with respect to $s$ on $[0,t]$ we find the assertion.\\
Now we prove the claim of the theorem. Let $\mu^1,\mu^2\in \rC$ be two solutions to (\ref{eq:FP2}) with the same initial condition $\mu^0$; then for every $t$ (in the sequel we can consider the supremum over $\phi\in K_\infty \subset C_c^{\infty}$ such that $\VV\phi\VV_\infty\leq 1$ as in the previous remark)
\begin{align}
\nonumber  d_{\mathrm{TV}}(\mu^1_t,\mu^2_t)&=\sup_{\VV\phi\VV_\infty\leq 1}\left\vert\int_0^t\left[\langle\mu^1_s,b(\mu^1_s)
\partial_v T_{t-s}\phi\rangle-\langle\mu^2_s,b(\mu^2_s)\partial_v T_{t-s}\phi\rangle\right]\ud s\right\vert\\
\label{eq:term1}                   &\leq\sup_{\VV\phi\VV_\infty\leq 1}\left\vert\int_0^t\langle\mu^1_s,\left[b(\mu^1_s)-b(\mu^2_s)\right]\partial_v T_{t-s}\phi\rangle\ud s\right\vert\\
\label{eq:term2}                   &\phantom{\leq}+\sup_{\VV\phi\VV_\infty\leq 1}\left\vert\int_0^t\langle\mu^1_s-\mu^2_s,b(\mu^2_s)\partial_v 
T_{t-s} \phi\rangle\ud s\right\vert\ .
\end{align}
The function $b(\mu^2_s)\partial_v T_{t-s}\phi$ is bounded and measurable and we have the estimate (cf. \eqref{reg1})
\begin{equation*}
  \left\VV \partial_v T_{t-s} \phi\right\VV_\infty\leq\frac{C}{\sqrt{t-s}}\VV\phi\VV_\infty;
\end{equation*}
we can thus bound the term (\ref{eq:term2}) by
\begin{equation*}
  \sup_{\VV\phi\VV_\infty\leq 1}\int_0^td_{\mathrm{TV}}(\mu_s^1,\mu_s^2)\left\VV b(\mu^2_s)\partial_v T_{t-s } \phi\right\VV_\infty\ud s\leq \left(\VV\theta\VV_\infty+  
  \VV\lambda \VV_\infty \right)\int_0^t\frac{C}{\sqrt{t-s}} d_{\mathrm{TV}}(\mu^1_s,\mu^2_s)\ud s\ .
\end{equation*}
Similarly, (\ref{eq:term1}) is bounded by
\begin{equation*}
  \sup_{\VV\phi\VV_\infty\leq 1}\int_0^t\left\VV \left(b(\mu_s^1)-b(\mu_s^2)\right)\partial_v T_{t-s}\phi\right\VV_\infty\ud s\leq \VV\theta\VV_\infty\int_0^t\frac{C}{\sqrt{t-s}}d_{\mathrm{TV}}(\mu_s^1,\mu_s^2)\ud s\ .
\end{equation*}
An application of a generalized version of Gronwall's lemma (see, for instance, \cite[Section 1.2.1]{henry1981}) yields that $\mu_t^1 = \mu_t^2$ for every $t$.
\end{proof}

\section{The laws of the empirical measures}
\label{section:Q_tight}
We denote by $\bQ^N$ the law of $S^N$ on $\rC$ (we are considering each $S^N$ as a r.v. with values in $\rC$). As explained in the introduction, we need to show tightness of the family $\bQ^N$.
\begin{theorem}
\label{thm:tight1}
  The sequence $\left\{\bQ^N\right\}_{N\in\bN}$ is tight in $\rC$.
\end{theorem}
\begin{proof}
Fix any $(x_0,v_0)\in D\times\bT$ and consider the set
\begin{multline*}
  \sK_{M,R}=\bigg\{\mu\in \rC\colon \\ \sup_{t\in[0,T]}\int_{D\times\bT}d_{D\times\bT}\big((x_0,v_0),(x,v)\big)\mu_t\left(\ud x,\ud v\right)\leq M,\ \int_0^T\int_0^T\frac{\sW_1\left(\mu_t,\mu_s\right)^p}{\lvert t-s\rvert^{1+\alpha p}}\ud t\ud s\leq R\bigg\}
\end{multline*}
where we choose $\alpha\in(0,1)$ and $p\geq 1$ such that $\alpha p>1$.\\
We show that $\sK_{M,R}$ is relatively compact in $\rC$. Let $B_{(x_0,v_0)}(r)$ denote the open ball with radius $r$ and center $(x_0,v_0)$ in $D\times\bT$. Then for $\mu\in\sK_{M,R}$ and $t\in[0,T]$
\begin{equation*}
  \mu_t\left(B_{(x_0,v_0)}(r)^\complement\right)\leq\frac{1}{r}\int_{D\times\bT}d_{D\times\bT}\big((x_0,v_0),(x,v)\big)\mu_t(\ud x,\ud v)\leq \frac{M}{r}\ .
\end{equation*}
Therefore for every $e>0$ and for every $t\in[0,T]$ we can find $r=r(e,t)$ such that
\begin{equation}
\label{eq:cpt1}
  \mu_t\left(B_{(x_0,v_0)}(r)\right)>1-e
  \end{equation}
for every $\mu\in\sK_{M,R}$.\\
By the Sobolev embedding theorem, if $\beta<(\alpha p-1)/p$, we have that, for any Lipschitz continuous function $\phi$ on $D\times\bT$ and any $t,s\in[0,T]$,
\begin{equation*}
  \lvert\langle\mu_t,\phi\rangle-\langle\mu_s,\phi\rangle\rvert\leq C\lvert t-s\rvert^\beta\left(\int_0^T\int_0^T\frac{\left\vert\langle\mu_t,\phi\rangle-\langle\mu_s,\phi\rangle\right\vert^p}{\lvert t-s\rvert^{1+\alpha p}}\ud t\ud s\right)^{{\frac{1}{p}}}
\end{equation*}
so that, thanks to Kantorovich-Rubinstein characterization of the $1$-Wasserstein distance, we can take the supremum over Lipschitz functions on $D\times\bT$ with Lipschitz seminorm bounded by $1$ on both sides of the previous inequality, obtaining
\begin{equation*}
  \sup_{t\neq s}\frac{\sW_1(\mu_t,\mu_s)}{\lvert t-s\rvert^\beta}\leq C\left(\int_0^T\int_0^T\frac{\sW_1(\mu_t,\mu_s)^p}{\lvert t-s\rvert^{1+\alpha p}}\ud t\ud s\right)^{{\frac{1}{p}}}\ .
\end{equation*}
Therefore the collection of measures $\sK_{M,R}$ is equicontinuous; this together with (\ref{eq:cpt1}) implies relative compactness by the Ascoli-Arzel\`a theorem.

To show tightness we now compute
\begin{align*}
  \bQ^N\left(\sK_{M,R}^\complement\right)&=\bP\left(S^N\in\sK_{M,R}^\complement\right)\\
  &\leq\bP\left(\sup_{t\in[0,T]}\int_{D\times\bT}d_{D\times\bT}\big((x_0,v_0),(x,v)\big) S^N_t\left(\ud x,\ud v\right)>M\right)\\
&\phantom{aaaaaaaaaa}+\bP\left(\int_0^T\int_0^T\frac{\sW_1\left(S^N_t,S^N_s\right)^p}{\lvert t-s\rvert^{1+\alpha p}}\ud t\ud s>R\right).
\end{align*}
For the first term we have
\begin{align*}
  \bP\bigg(\sup_{t\in[0,T]}\int_{D\times\bT}&d_{D\times\bT}\big((x_0,v_0),(x,v)\big) S^N_t\left(\ud x,\ud v\right)>M\bigg)\\
  &\le \frac{1}  {M}\bE\left[\left\vert\sup_{t\in[0,T]}\int_{D\times\bT}d_{D\times\bT}\big((x_0,v_0),(x,v)\big) S^N_t\left(\ud x,\ud v\right)\right\vert\right]\\
  &\le\frac{1}{MN}\bE\left[\sup_{t\in[0,T]}\sum_{j=1}^Nd_{D\times\bT}\left((x_0,v_0),\left( X^{i}_0,V^{i,N}_t\pmod{2}\right)\right)\right]\\
&\le\frac{1}{M}+\frac{1}{MN}\sum_{i=1}^N\bE\left[\sup_{t\in[0,T]}\left\vert\left( X^i_0-x_0,V^{i,N}_t\right)\right\vert\right]\\
  &\le \frac{1}{M}+\frac{C }{MN}\sum_{i=1}^N\bE\Bigg[\left\vert\left(X^i_0-x_0,V_0^{i,N}\right)\right\vert \\
  &\phantom{\frac{1}{MN}\sum_{i=1}^N\bE\Bigg[}+ \int_0^t \sup_{t \in [0,T]} \left\vert\lambda_2\left(V_s^{i,N}\right)+\frac{1}{N}\sum_{j=1}^N g_2\left(X^i_0,V^{i,N}_s,X_0^j,V_s^{j,N}\right)\right\vert\ud s\\
  &\phantom{\lesssim\frac{1}{MN}\sum_{i=1}^N\bE \sup_{t \in [0,T]}\Bigg[\lesssim}+\sup_{t \in [0,T]}\left\vert\int_0^t\sigma_2^\epsilon\left(V_s^{i,N}\right)\ud B^i_s\right\vert\Bigg]
  \leq\frac{C}{M}
\end{align*}
for   a certain constant $C=C(\hat\lambda,\theta,\sigma^\epsilon,T)$, thanks to the Burkholder-Davis-Gundy inequality, the boundedness of $\lambda_2$, $g_2$ and $\sigma_2^\epsilon$ and the fact that $\nu\times\tilde\rho_0\rL_\bT$ has finite first moment.\\
For the second term
\begin{align*}
  \bP\left(\int_0^T\int_0^T\frac{\sW_1\left(S^N_t,S^N_s\right)^p}{\lvert t-s\rvert^{1+\alpha p}}\ud t\ud s>R\right)&\leq\frac{1}{R}\int_0^T\int_0^T\frac{\bE\left[\sW_1\left(S^N_t,S^N_s\right)^p\right]}{\lvert t-s\rvert^{1+\alpha p}}\ud t\ud s .
\end{align*}
Let $\phi$ be a Lipschitz function on $D\times\bT$ with Lipschitz constant $K_\phi\leq 1$. Then
\begin{equation*}
  \left\vert\left\langle S_t^N,\phi\right\rangle-\left\langle S_s^N,\phi\right\rangle\right\vert\leq\frac{1}{N}\sum_{i=1}^Nd_\bT\left( V^{i,N}_t,V^{i,N}_s\right)
\end{equation*}
so that, by the Kantorovich characterization of the $1$-Wasserstein distance and by H\"older's inequality,
\begin{equation*}
  \bE\left[\sW_1\left(S^N_t,S^N_s\right)^p\right]\leq\frac{1}{N}\sum_{i=1}^N\bE\left[d_\bT\left( V^{i,N}_t,V^{i,N}_s\right)^p\right] .
\end{equation*}
Recalling (\ref{eq:metric}), (\ref{eq:dV2}) and notation (\ref{eq:notation_b}), we can write
\begin{equation*}
  d_\bT\left( V^{i,N}_t,V^{i,N}_s\right)\leq\left\vert V^{i,N}_t-V^{i,N}_s\right\vert\leq\int_s^t\left\vert b(S^N_r)(X^i_0,V^{i,N}_r)\right\vert\ud r+\left\vert\int_s^t\sigma^\epsilon_2(V^{i,N}_r)\ud B^i_r\right\vert
\end{equation*}
so that
\begin{equation*}
 \bE\left[d_\bT\left( V^{i,N}_t,V^{i,N}_s\right)^p\right\vert\leq C^\prime\vert t-s\vert^{\nicefrac{p}{2}}
\end{equation*}
for a suitable constant $C^\prime=C^\prime(\hat\lambda,\theta,\sigma^\epsilon)$, again by boundedness of the coefficients and the Burkholder-Davis-Gundy inequality. Choosing $p>2$ and $\alpha$ such that $\alpha p<\nicefrac{p}{2}-1$ we find
\begin{equation*}
  \bP\left(\int_0^T\int_0^T\frac{\sW_1\left(S^N_t,S^N_s\right)^p}{\lvert t-s\rvert^{1+\alpha p}}\ud t\ud s>R\right)\leq\frac{C^\prime}{R}\ .
\end{equation*}
For 
any $e>0$ we can now choose $M$ and $R$ so that $\bQ^N\left(\sK_{M,R}^\complement\right)<e$, concluding the proof.
\end{proof}
{ An alternative approach to prove theorem \ref{thm:tight1} could be based on tightness results from \cite[Chapters I and II]{Szn89}, using the boundedness of the coefficients and the interchangeability of the $V^{i,N}$. However, the above direct proof  can be applied to more general situations as well.}
\section{The limit PDE: existence and convergence}
\label{section:PDE_existence}
\subsection{Density Estimates}
\label{subsection:estimates}
Recall that the empirical measure $S_{t}^{N}=\frac{1}{N}\sum_{i=1}^{N}%
\delta_{\left(\xi^{i},V_{t}^{i,N}\pmod*{2}\right)}$ satisfies%
\begin{align*}
\left\langle S_{t}^{N},\phi\right\rangle  &  =\left\langle S_{0}^{N}%
,\phi\right\rangle +\int_{0}^{t}\left\langle S_{r}^{N},\lambda_{2}\partial
_{v}\phi\right\rangle \ud r\\
&  +\int_{0}^{t}\left\langle S_{r}^{N},\left\langle S_{r}^{N},g_{2}\left(
x,v,\cdot,\cdot\right)  \right\rangle \partial_{v}\phi\right\rangle
\ud r+\int_{0}^{t}\left\langle S_{r}^{N},\frac{\left(  \sigma_{2}^{\epsilon
}\right)  ^{2}}{2}\partial_{v}^{2}\phi\right\rangle \ud r+M_{t}^{N,\phi}%
\end{align*}
where%
\[
M_{t}^{N,\phi}=\frac{1}{N}\sum_{i=1}^{N}\int_{0}^{t}\sigma_{2}^{\epsilon
}\left(  V_{r}^{i,N}\right)  \partial_{v}\phi\left(  X_{0}^{i},V_{r}%
^{i,N}\right)  \ud B_{r}^{i}.
\]
We consider a smooth probability density $\gamma : \bT \to \bR$ defined as follows:
$$
\gamma (v) =
\begin{cases}
 c \; \exp\left(-\frac {1}{ \frac{1}{4} - (d_{\bT}(v,0))^2  }\right) \cdot (\frac{1}{4} - d_{\bT}(v,0)^2)^2,\;\;\; \text{if } \;  d_{\bT}(v,0)
 < 1/2,
\\
0 \;\; \text{otherwise.}
\end{cases}
$$
 and 
introduce a correspondent  family of mollifiers $\gamma_{N}\left(  v\right)  =\alpha
_{N}^{-1}\gamma\left(  \alpha_{N}^{-1}v\right)  $ {on $\bT$}. Note that there exists $C>0$ such that 
\[
\left\vert \gamma^{\prime}\left(  w\right)  w\right\vert \leq C\gamma\left(
w\right), \;\; w \in \bT.  
\]
Concerning the positive scaling factor, we assume that $\alpha_N\to 0$ as $N\to\infty$ and
\[
\alpha_{N}^{-3}\leq N.
\]
Consider the empirical density
\begin{align*}
u_{t}^{N}\left(  v\right)  :&=\frac{1}{N}\sum_{i}\int_{D\times\bT}\gamma_{N}\left(v-v^{\prime}\right)  \delta_{\left(X_{0}^{i},V_{t}^{i,N}\pmod*{2}\right)}\left(  \ud x^{\prime},\ud v^{\prime}\right)\\
&=\frac{1}{N}\sum_{i}\gamma_{N}\left(  v-V_{t}^{i,N}\pmod{2}\right)
\\
&=\int_{D\times\bT}\gamma_{N}\left(  v-v^{\prime}\right)  S_{t}^{N}\left(  \ud x^{\prime},\ud v^{\prime}\right)\  
\end{align*}
(where sums and differences are understood on $\bT$, i.e. for $v_1,v_2\in\bT$, $v_1\pm v_2=(v_1\pm_\bR v_2)\pmod*{2}\in\bT$). It satisfies
\begin{align*}
\ud u_{t}^{N}\left(  v\right)   &  = -\langle S^N_t,\lambda_2\partial_v\gamma_N(v-\cdot)\rangle\ud t\\
&-\big\langle S^N_t,\langle S^N_t,g_2(x^\prime,v^\prime,\cdot,\cdot)\rangle\partial_v\gamma_N(v-v^\prime)\big\rangle\ud t \\
&+\langle S^N_t,\frac{(\sigma_2^\epsilon)^2}{2}\partial^2_v\gamma_N(v-\cdot)\rangle\ud t+\ud \overline{M}_{t}^{N}\left(  v\right)
\end{align*}
where
\[
\overline{M}_{t}^{N}\left(  v\right)  =-\frac{1}{N}\sum_{i=1}^{N}\int_{0}^{t}\partial
_{v}\gamma_{N}\left(  v-V_{r}^{i,N}\right)  \sigma_{2}^{\epsilon}\left(
V_{r}^{i,N}\right)  \ud B_{r}^{i}
\]
and according to \eqref{eq:notaz} we write
\begin{align*}
 \left\langle S_{t}^{N},\left( \left\langle S_{t}^{N}
,g_{2}\left(  x^\prime,v^\prime,\cdot,\cdot\right)  \right\rangle \right)  \partial
_{v}\gamma_{N}\left(  v- v'\right)  \right\rangle
\\
=   \int_{D\times\bT}\partial_v\gamma_N(v-v^\prime)\int_{D\times\bT}g_2(x^\prime,v^\prime,y,w)S^N_t(\ud y,\ud w)S^N_t(\ud x^\prime,\ud v^\prime).
\end{align*}
In the next lemma we will use that
$\sigma_{2}^{\epsilon}\left(  v\right)  $ is differentiable with bounded
derivative and
\[
0<\epsilon\leq\frac{\left(  \sigma_{2}^{\epsilon}\left(  v\right)
\right)  ^{2}}{2}\leq C.
\]

\begin{lemma}
\label{lem:51}
There exists a constant $C_{\epsilon}>0$ such that%
\[
\sup_{t\in\left[  0,T\right]  }E\int_\bT\left\vert u_{t}^{N}\left(  v\right)
\right\vert ^{2}\ud v+E\int_{0}^{T}\int_\bT\left\vert \partial_{v}u_{t}^{N}\left(
v\right)  \right\vert ^{2}\ud v\ud t\leq C_{\epsilon}%
\]
for every $N\in\bN$.
\end{lemma}

\begin{proof}
\textbf{Step 1} (energy identity). One has by It\^o's formula, integrating by parts,
\begin{align*}
\frac{1}{2}\ud\int_{\bT}&\left\vert u_{t}^{N}\left(  v\right)  \right\vert ^{2}\ud v 
=\int_{\bT}\frac{\left(  \sigma_{2}^{\epsilon}\left(  v\right)  \right)  ^{2}}%
{2}u_{t}^{N}\left(  v\right)  \partial_{v}^{2}u_{t}^{N}\left(  v\right)
\ud v\ud t\\
&  +\int_{\bT}\left\langle S_{t}^{N},\left(  \lambda_{2}+\left\langle S_{t}%
^{N},g_{2}\left(  x^\prime,v^\prime,\cdot,\cdot\right)  \right\rangle \right)  \gamma_{N}\left(  v-\cdot\right)  \right\rangle \partial_{v}u_{t}^{N}\left(
v\right)  \ud v\ud t\\
&  -\int_{\bT} R_{t}^{N}\left(  v\right)  \partial_{v}u_{t}^{N}\left(  v\right)
\ud v\ud t+\int_\bT u_{t}^{N}\left(  v\right)  \ud \overline{M}_{t}^{N}\left(  v\right)  \ud v+\frac
{1}{2}\int_\bT \ud\left[  \overline{M}^{N}\left(  v\right)  \right]  _{t}\ud v
\end{align*}
where 
\[
R_{t}^{N}\left(  v\right)  =\partial_{v}\int_{D\times\bT}\gamma_{N}\left(  v-v^{\prime
}\right)  \left[ - \frac{\left(  \sigma_{2}^{\epsilon}\left(  v\right)
\right)  ^{2}}{2}+\frac{\left(  \sigma_{2}^{\epsilon}\left(  v^{\prime
}\right)  \right)  ^{2}}{2}\right]  S_{t}^{N}\left(  \ud x^{\prime},\ud v^{\prime
}\right)\ ,
\]
and we write $\int_\bT u_{t}^{N}\left(  v\right)  \ud \overline{M}_{t}^{N}\left(  v\right)
\ud v$ for
\[
\int_\bT u_{t}^{N}\left(  v\right)  \ud\overline{M}_{t}^{N}\left(  v\right)  \ud v=\frac{1}{N}%
\sum_{i=1}^{N}\left(  \int_\bT\partial_{v}\gamma_{N}\left(  v-V_{t}^{i,N}\right)
u_{t}^{N}\left(  v\right)  \ud v\right)  \sigma_{2}^{\epsilon}\left(  V_{t}%
^{i,N}\right)  \ud B_{t}^{i}%
\]
and $\int_\bT \ud\left[  \overline{M}^{N}\left(  v\right)  \right]  _{t}\ud v$ for%
\[
\int_\bT \ud\left[  \overline{M}^{N}\left(  v\right)  \right]  _{t}\ud v=\frac{1}{N^{2}}\sum
_{i=1}^{N}\int_\bT\left\vert \partial_{v}\gamma_{N}\left(  v-V_{t}^{i,N}\right)
\right\vert ^{2}\ud v\left\vert \sigma_{2}^{\epsilon}\left(  V_{t}^{i,N}\right)
\right\vert ^{2}\ud t.
\]
\textbf{Step 2} (deterministic terms). Using the assumptions on $\sigma
_{2}^{\epsilon}\left(  v\right)  $, one has integrating by parts
\[
\int_\bT\frac{\left(  \sigma_{2}^{\epsilon}\left(  v\right)  \right)  ^{2}}%
{2}u_{t}^{N}\left(  v\right)  \partial_{v}^{2}u_{t}^{N}\left(  v\right)
\ud v\leq-\frac{\epsilon}{2}\int_\bT\left\vert \partial_{v}u_{t}^{N}\left(  v\right)
\right\vert ^{2}\ud v+C_{\epsilon}\int_\bT\left\vert u_{t}^{N}\left(  v\right)
\right\vert ^{2}\ud v
\]
Since (due to the boundedness of $\lambda_{2}$ and $g_{2}$)
\[
\left\vert \lambda_{2}+\left\langle S_{t}^{N},g_{2}\left(  x,v,\cdot
,\cdot\right)  \right\rangle \right\vert \leq C
\]
one has
\[
\left\vert \left\langle S_{t}^{N},\left(  \lambda_{2}+\left\langle S_{t}%
^{N},g_{2}\left(  x^\prime,v^\prime,\cdot,\cdot\right)  \right\rangle \right)  \gamma_{N}\left(  v-\cdot\right)  \right\rangle \right\vert \leq Cu_{t}^{N}\left(
v\right)  .
\]
Therefore, $\bP$-a.s.,
\begin{multline*}
\int_\bT\left\langle S_{t}^{N},\left(  \lambda_{2}+\left\langle S_{t}^{N}%
,g_{2}\left(  x^\prime,v^\prime,\cdot,\cdot\right)  \right\rangle \right)  \gamma_{N}\left(
v-\cdot\right)  \right\rangle \partial_{v}u_{t}^{N}\left(  v\right)
\ud v\\ \leq\frac{\epsilon}{4}\int_\bT\left\vert \partial_{v}u_{t}^{N}\left(
v\right)  \right\vert ^{2}\ud v+C_{\epsilon}\int_\bT\left\vert u_{t}^{N}\left(
v\right)  \right\vert ^{2}\ud v.
\end{multline*}
We have got that $\bP$-a.s.
\begin{multline*}
\frac{1}{2}\ud\int_\bT\left\vert u_{t}^{N}\left(  v\right)  \right\vert ^{2}%
\ud v+\frac{\epsilon}{4}\int_\bT\left\vert \partial_{v}u_{t}^{N}\left(
v\right)  \right\vert ^{2}\ud v\ud t\leq C_{\epsilon}\int_\bT\left\vert u_{t}^{N}\left(
v\right)  \right\vert ^{2}\ud v\\
-\int_\bT R_{t}^{N}\left(  v\right)  \partial_{v}u_{t}^{N}\left(  v\right)
\ud v\ud t+\int_\bT u_{t}^{N}\left(  v\right)  \ud\overline{M}_{t}^{N}\left(  v\right)  \ud v+\frac
{1}{2}\int_\bT \ud\left[  \overline{M}^{N}\left(  v\right)  \right]  _{t}\ud v.
\end{multline*}
Finally, using also the assumption $\left\vert \gamma^{\prime}\left(
w\right)  w\right\vert \leq C\gamma\left(  w\right)  $,
\begin{align*}
\left\vert R_{t}^{N}\left(  v\right)  \right\vert  &  \leq C\int_{D\times\bT}\left[\left\vert
\partial_{v}\gamma_{N}\left(  v-v^{\prime}\right)  \right\vert \left\vert
v-v^{\prime}\right\vert+\gamma_N(v-v^\prime)\right] S_{t}^{N}\left(  \ud x^{\prime},\ud v^{\prime}\right)  \\
&  =C\int_{D\times\bT}\left[\alpha_{N}^{-1}\left\vert \gamma^{\prime}\left(  \alpha_{N}%
^{-1}\left(  v-v^{\prime}\right)  \right)  \right\vert \alpha_{N}%
^{-1}\left\vert v-v^{\prime}\right\vert +\gamma_N(v-v^\prime)\right]S_{t}^{N}\left(  \ud x^{\prime
},\ud v^{\prime}\right)  \\
&  \leq C^{\prime}\int_{D\times\bT}\alpha_{N}^{-1}\gamma\left(  \alpha
_{N}^{-1}\left(  v-v^{\prime}\right)  \right) S_{t}^{N}\left(
\ud x^{\prime},\ud v^{\prime}\right)  \\
&  =C^{\prime}u_{t}^{N}\left(  v\right)  
\end{align*}
which yields
\[
\left\vert \int_\bT R_{t}^{N}\left(  v\right)  \partial_{v}u_{t}^{N}\left(
v\right)  \ud v\right\vert \leq\frac{\epsilon}{8}\int_\bT\left\vert
\partial_{v}u_{t}^{N}\left(  v\right)  \right\vert ^{2}\ud v+C_{\epsilon}%
\int_\bT\left\vert u_{t}^{N}\left(  v\right)  \right\vert ^{2}\ud v.
\]
This implies
\begin{multline*}
\frac{1}{2}\ud\int_\bT\left\vert u_{t}^{N}\left(  v\right)  \right\vert ^{2}%
\ud v+\frac{\epsilon}{8}\int_\bT\left\vert \partial_{v}u_{t}^{N}\left(
v\right)  \right\vert ^{2}\ud v\ud t\\ \leq C_{\epsilon}\int_\bT\left\vert u_{t}^{N}\left(
v\right)  \right\vert ^{2}\ud v
+\int_\bT u_{t}^{N}\left(  v\right)  \ud\overline{M}_{t}^{N}\left(  v\right)  \ud v+\frac{1}%
{2}\int_\bT \ud\left[  \overline{M}^{N}\left(  v\right)  \right]  _{t}\ud v\ ,
\end{multline*}
$\bP$ almost surely.

\textbf{Step 3} (martingale terms and conclusion). It remain to handle the sum
$$\int_\bT u_{t}^{N}\left(  v\right)  \ud\overline{M}_{t}^{N}\left(  v\right)  \ud v+\frac{1}{2}\ud\left[  \overline{M}^{N}\left(  v\right)  \right]  _{t}\ .$$ The term $\int_\bT u_{t}^{N}\left(  v\right)  \ud\overline{M}_{t}^{N}\left(  v\right)  \ud v$ is a martingale, hence
it has mean zero. Indeed, for every $N$ and $i=1,...,N$,%
\begin{align*}
\bE\int_{0}^{T}&\left\vert \left(  \int_\bT\partial_{v}\gamma_{N}\left(
v-V_{t}^{i,N}\right)  u_{t}^{N}\left(  v\right)  \ud v\right)  \sigma
_{2}^{\epsilon}\left(  V_{t}^{i,N}\right)  \right\vert ^{2}\ud t\\
& \leq C\bE\int_{0}^{T}\left\vert \int_\bT\partial_{v}\gamma_{N}\left(
v-V_{t}^{i,N}\right)  u_{t}^{N}\left(  v\right)  \ud v\right\vert ^{2}\ud t\\
& \leq C_{N}\bE\int_{0}^{T}\left\vert \int_\bT u_{t}^{N}\left(  v\right)
\ud v\right\vert ^{2}\ud t=C_{N}T
\end{align*}
(because $\int_\bT u_{t}^{N}\left(  v\right)  \ud v=1$). As to the corrector, we have%
\begin{align*}
\int_\bT\left\vert \partial_{v}\gamma_{N}\left(  v-V_{r}^{i,N}\right)  \right\vert
^{2}\ud v  & =\int_\bT\left\vert \partial_{v}\gamma_{N}\left(  v\right)  \right\vert
^{2}\ud v\\
& =\alpha_{N}^{-2}\alpha_{N}^{-2}\int_\bT\left\vert \gamma^{\prime}\left(
\alpha_{N}^{-1}v\right)  \right\vert ^{2}\ud v=C\alpha_{N}^{-3}%
\end{align*}
where $C=\int_\bT\left\vert \gamma^{\prime}\left(  v\right)  \right\vert ^{2}\ud v$.
Hence, $\bP$-a.s.,
\begin{align*}
\int_\bT\left[  \overline{M}^{N}\left(  v\right)  \right]  _{t}dv  & =\frac{1}{N^{2}}%
\sum_{i=1}^{N}\int_{0}^{t}\int_\bT\left\vert \partial_{v}\gamma_{N}\left(
v-V_{r}^{i,N}\right)  \right\vert ^{2}\ud v\left\vert \sigma_{2}^{\epsilon
}\left(  V_{r}^{i,N}\right)  \right\vert ^{2}\ud r\\
& \leq C\frac{1}{N^{2}}\sum_{i=1}^{N}\int_{0}^{t}\left(  \int_\bT\left\vert
\partial_{v}\gamma_{N}\left(  v-V_{r}^{i,N}\right)  \right\vert ^{2}\ud v\right)
\ud r=C\frac{\alpha_{N}^{-3}}{N}t\leq Ct
\end{align*}
under the assumption $\alpha_{N}^{-3}\leq N$.\\
Using the assumption that the law of the initial data $\eta^i$ has an $L^2$ density, it is not difficult to show that the $L^2$ norm of $u^N_0$ is bounded uniformly with respect to $N$. { To this purpose let us recall that we denote by $\tilde\rho$ the density of each $\eta^i$. Using also standard property of convolutions we get:}
\begin{multline*}
\bE \int_{\bT}   |u_{0}^{N}(v)|^2 \ud v  
=
\int_{\bT}  \bE \Big | \frac{1}{N}\sum_{i}\gamma_{N} \left(  v- \eta^{i}\right) \Big|^2 \ud v 
\\
= \frac{1}{N^2} \frac{1}{\alpha_N^2} \int_{\bT}   \,  \sum_{i} \bE \big |\gamma \left(  \frac{ v- \eta^{i}}{\alpha_N}\right) \big|^2  \ud v =
\frac{1}{N} \frac{1}{\alpha_N^2} \int_{\bT}  \ud v   \int_{\bT}  \Big |\gamma \left(  \frac{ v- w}{\alpha_N}\right) \Big|^2 \tilde\rho(w) \ud w   
\\
\le \| \tilde\rho\|_{L^1(\bT)}
\, \frac{1}{N} \frac{1}{\alpha_N^2}    \int_{\bT}  \Big |\gamma \left(  \frac{ v}{\alpha_N}\right) \Big|^2  \ud v
= 
\| \tilde\rho\|_{L^1(\bT)}
\, \frac{1}{N} \frac{1}{\alpha_N}    \int_{\bT}  \Big |\gamma \left(  { v'}\right) \Big|^2  \ud v' \le C,
\end{multline*}
where $C>0$ is independent of $N$. 
We can therefore take expectation and apply Gronwall's lemma, thus deducing the claim from the results of the two previous steps.
\end{proof}

\begin{lemma}%
\label{lem:52}
There exists $\alpha>0$ small enough such that
\[
\bE\int_{0}^{T}\int_{0}^{T}\frac{\left\Vert u_{t}^{N}-u_{s}^{N}\right\Vert
_{H^{-2}}^{2}}{\left\vert t-s\right\vert ^{1+2\alpha}}\ud s\ud t\leq C_{\epsilon}.
\]
where $H^{-2}=H^{-2}(\bT)$.
\end{lemma}

\begin{proof}%
Arguing as in Lemma \ref{lem:51}, we have, $\bP$-a.s., for any $0\le s \le t \le T$, $\phi \in H^2(\bT),$
\begin{align*}
\int_{\bT} [u_{t}^{N}(v)-u_{s}^{N}(v) ]\phi(v) dv & =\int_{s}^{t}  dr \int_{\bT}  \frac{\left(
\sigma_{2}^{\epsilon}\left(  v\right)  \right)  ^{2}}{2}u_{r}^{N}\left(
v\right)  \partial_v^2 \phi(v)  \ud v\\
& \int_{s}^{t} dr \int_{\bT} \left\langle S_{r}^{N},\left(  \lambda_{2}+\left\langle
S_{r}^{N},g_{2}\left(  x^\prime,v^\prime,\cdot,\cdot\right)  \right\rangle \right)
\gamma_{N}\left(  v- v'\right)  \right\rangle \partial_v \phi(v)
\ud v\\
& - \int_{s}^{t}dr \int_{\bT} R_{r}^{N}\left(  v\right) \partial_v \phi(v) \ud v
\\
& + \int_{\bT}  [\overline{M}_{t}^{N}\left(  v\right)  -\overline{M}_{s}^{N}\left(  v\right) ] \phi(v) dv
\end{align*}
where $R_{t}^{N}\left(  v\right)  $ and $\overline{M}_{t}^{N}\left(  v\right)  $ are
given in Step 1 of the previous lemma. Then (using the same inequalities
proved above in Step 2 of the previous lemma)
{\begin{align*}
\left\Vert u_{t}^{N}-u_{s}^{N}\right\Vert _{H^{-2}}^{2}  \leq 3\times C\left(t-s\right)  \int_{s}^{t}\left\Vert u_{r}^{N}\right\Vert _{L^{2}}^{2}\ud r+C\left\Vert \overline{M}_{t}^{N}-\overline{M}_{s}^{N}\right\Vert _{H^{-2}}^{2}.
\end{align*}
}
It is sufficient (because of the claim of the previous lemma)\ to prove that%
\[
\bE\int_{0}^{T}\int_{0}^{T}\frac{\left\Vert \overline{M}_{t}^{N}-\overline{M}_{s}^{N}\right\Vert
_{H^{-2}}^{2}}{\left\vert t-s\right\vert ^{1+2\alpha}}\ud s\ud t\leq C_{\epsilon}.
\]
Recall that
\[
\overline{M}_{t}^{N}\left(  v\right)  =\partial_{v}\left(  \frac{1}{N}\sum_{i=1}^{N}%
\int_{0}^{t}\gamma_{N}\left(  v-V_{r}^{i,N}\right)  \sigma_{2}^{\epsilon
}\left(  V_{r}^{i,N}\right)  \ud B_{r}^{i}\right)  .
\]
Then
\begin{align*}
\bE\left[  \left\Vert \overline{M}_{t}^{N}-\overline{M}_{s}^{N}\right\Vert _{H^{-2}}^{2}\right]    &
\leq C\bE\left[  \left\Vert \frac{1}{N}\sum_{i=1}^{N}\int_{s}^{t}\gamma_{N}\left(  v-V_{r}^{i,N}\right)  \sigma_{2}^{\epsilon}\left(  V_{r}%
^{i,N}\right)  \ud B_{r}^{i}\right\Vert _{L^{2}}^{2}\right]  \\
& =\frac{C}{N^{2}}\int_\bT \bE\left[  \left\vert \sum_{i=1}^{N}\int_{s}^{t}%
\gamma_{N}\left(  v-V_{r}^{i,N}\right)  \sigma_{2}^{\epsilon}\left(
V_{r}^{i,N}\right)  \ud B_{r}^{i}\right\vert ^{2}\right]  \ud v\\
& =\frac{C}{N^{2}}\int_\bT\sum_{i=1}^{N}\bE\left[  \int_{s}^{t}\left\vert \gamma_{N}\left(  v-V_{r}^{i,N}\right)  \sigma_{2}^{\epsilon}\left(  V_{r}%
^{i,N}\right)  \right\vert ^{2}\ud r\right]  \ud v\\
& \leq\frac{C}{N^{2}}\int_\bT\sum_{i=1}^{N}\bE\left[  \int_{s}^{t}\left\vert
\gamma_{N}\left(  v-V_{r}^{i,N}\right)  \right\vert ^{2}dr\right]  \ud v\\
&=\frac{C}{N^{2}}\sum_{i=1}^{N}\int_{s}^{t}\bE\left[  \int_\bT\left\vert \gamma_{N}\left(
v-V_{r}^{i,N}\right)  \right\vert ^{2}\ud v\right]  \ud r\\
& =\frac{C}{N^{2}}\sum_{i=1}^{N}\int_{s}^{t}\int_\bT\left\vert \gamma_{N}\left(
v\right)  \right\vert ^{2}\ud v\ud r\leq C\frac{\alpha_{N}^{-1}}{N}\left(
t-s\right)  \leq C\left(  t-s\right)
\end{align*}
where we have used the estimate $\int_\bT\left\vert \gamma_{N}\left(  v\right)\right\vert ^{2}\ud v\leq C\alpha_{N}^{-1}$ and the assumption $\alpha_{N}^{-3}\leq N$. The proof is complete.
\end{proof}

Now let $\bQ_{u^{N}}$ denote the law of the process $u^{N}$. From the previous two lemmas, we deduce that the family
$\left(\bQ_{u^N}\right)$ is tight in
\[
L^{2}\left(  0,T;L^{2}\left(  \bT \right)  \right)
\]
due to a generalized version of Aubin-Lions lemma, which claims that the
space
\[
L^{2}\left(  0,T;W^{1,2}\left(  \bT \right)  \right)  \cap W^{\alpha,2}\left(  0,T;H^{-2}\left(  \bT \right)  \right)
\]
is relatively compact in $L^{2}\left(  0,T;L^{2}\left(  \bT \right)  \right)  $, for $\alpha>0$ (cf. \cite{Sim87}).

\begin{remark} {Introducing  the mollifiers 
 $\gamma_{n}\left(  v\right)  =\alpha
_{n}^{-1}\gamma\left(  \alpha_{n}^{-1}v\right)  $, $n \ge 1$, with  
$
\left\vert \gamma^{\prime}\left(  w\right)  w\right\vert \leq C\gamma\left(
w\right), \;\; w \in \bT,  
$
  $\alpha_n\to 0$ as $n\to\infty$, and
 $ \alpha_{n}^{-3}\leq N,  $ and following the proof of Lemma
\ref{lem:51}, we can obtain  that there exists a constant $C_{\epsilon}>0$ 
such that
\begin{equation} \label{344}
\sup_{t\in\left[  0,T\right]  } \bE\int_\bT\left\vert u_{t}^{n, N}\left(  v\right)
\right\vert ^{2}\ud v+ \bE\int_{0}^{T}\int_\bT\left\vert \partial_{v}u_{t}^{n, N}\left(
v\right)  \right\vert ^{2}\ud v\ud t\leq C_{\epsilon},
\end{equation}
for every $n, \, N\in\bN$, where 
$
u_{t}^{n, N}(v ) 
 =\int_{D\times\bT}\gamma_n\left(  v-v^{\prime}\right)  S_{t}^{N}\left(  \ud x^{\prime},\ud v^{\prime}\right).  $

Now note that  given a Borel probability measure $\nu$ on $\bT$, if there exists $c>0$, such that,  for any $n \ge 1,$
\begin{equation} \label{dee}
\| \nu * \gamma_n \|_{L^2(\bT)} \le c,
\end{equation}
then $\nu \in L^2(\bT)$ and $\| \nu  \|_{L^2(\bT)} \le c$. Indeed, by \eqref{dee}, for any $\phi \in L^2(\bT)$, we have
$$
\Big | \int_{\bT} \phi (y) dy \int_{\bT}    \gamma_n (y- y') \nu (dy')\Big |
= 
\Big | \int_{\bT} \nu (dy')  \int_{\bT} \phi(y)   \gamma_n (y- y') dy\Big | \le c \| \phi \|_{L^2(\bT)}.
$$
Passing to the limit as $n \to \infty$ and using the Riesz  theorem we get the assertion.

Estimate \eqref{344} and the previous argument  could be used to  prove existence of solutions to \eqref{eq:FP2} in $\tilde X$   (see the next section) avoiding the previous Aubin-Lions lemma.
}
\end{remark}

\subsection{Convergence and existence of solutions}
\label{subsection:existence}
Set for notational convenience
\begin{equation}
\label{eq:xx}
\widetilde{X}=\left\{  \mu\in \rC\colon\pi_{v}\mu_{t}\ll\rL_{\bT}\; \text{with} \; \frac{\ud\left(\pi_{v}\mu_{t}\right)}{\ud \rL_\bT} \in L^2 (\bT),\;\text{ for a.e. }t\in\left[  0,T\right]  \right\} 
\end{equation}
where $\pi_{v}\mu_{t}$ is the marginal on the $v$-component of $\mu_{t}$:
\[
\int_\bT f\left(  v\right)  \left(  \pi_{v}\mu_{t}\right)  \left(  dv\right)
:=\int_{D\times\bT}f\left(  v\right)  \mu_{t}\left(  dx,dv\right), \;\;\; f\in C_{}\left(  \bT \right).
\]
\begin{lemma}
\label{lem:53}
The space $\tilde X$ is a Borel subset of $\rC$. 
\end{lemma}
\begin{proof}
 It is enough to show that $\Lambda = \left\{  \mu\in \mathrm{Pr}_{1}\left(D\times\bT\right)\colon\pi_{v}\mu\ll\rL_\bT\; \text{with} \; \frac{\ud\left(\pi_{v}\mu\right)}{\ud\rL_\bT} \in L^2 (\bT) \right\} 
$ is a Borel subset of $\mathrm{Pr}_{1}\left(D\times\bT\right).$ 

We consider  the continuous mapping $\sJ : \mathrm{Pr}_{1}\left(D\times\bT\right) \to \mathrm{Pr}_{1}\left(\bT\right)$ given by $\sJ \mu = \pi_{v}\mu$, for any $\mu \in \mathrm{Pr}_{1}\left(D\times\bT\right)$. 
If we prove that 
$$
\Gamma = \left\{  \mu\in \mathrm{Pr}_{1}\left(\bT\right)\colon\mu \ll\rL_\bT\; \text{with} \; \frac{\ud\mu}{\ud\rL_\bT}\in L^2 (\bT) \right \}
$$
 is Borel in $\mathrm{Pr}_{1}\left(\bT\right)$ then we get that $\Lambda = \sJ^{-1}(\Gamma)$ is Borel and this finishes the proof. Let us check the assertion on $\Gamma.$

Let $\mu \in \mathrm{Pr}_{1}\left(\bT\right).$  
Using the Riesz theorem  and the fact that $C(\bT)$ is dense in $L^2(\bT)$, we know that $\mu \in \Gamma$ if and only if 
 there exists $c>0$ such that 
\begin{equation} \label{d33}
  \big|\int_{\bT} f(y) \mu(dy) \big| \le c \| f\|_{L^2},\;\; \text{for all} \; f \in C(\bT)
\end{equation} 
(indeed if \eqref{d33} holds for  $\mu \in \mathrm{Pr}_{1}\left(\bT\right)$ then $\mu$ can be uniquely extended to a linear functional on $L^2(\bT)$). 
  Let us define, for integers $N \ge 1$, 
$$
\Gamma_N = \{\mu \in \mathrm{Pr}_{1}\left(\bT\right) \, :\, \text{ \eqref{d33} \ holds with $c$ replaced by $N$} \}
$$ 
It is easy to check that each $\Gamma_N$ is closed in $\mathrm{Pr}_{1}\left(\bT\right)$. We have $\Gamma = \bigcup_{N \ge 1} \Gamma_N $ and this shows that $\Gamma$ is Borel. 
\end{proof}

For any test function $\phi\in {\cal T}$ and any $\mu^0\in\mathrm{Pr}(D\times\bT)$ define on $\rC$ the functional
\begin{equation}
\label{eq:dff}
\Phi^{\mu^0}_{\phi}\left(  \mu\right)  =\sup_{t\in\left[  0,T\right]  }\left\vert
\left\langle \mu_{t},\phi\right\rangle -\left\langle \mu^{0},\phi\right\rangle
-\int_{0}^{t}\left\langle \mu_{s},b\left(  \mu_{s}\right)  \partial
_{v}\phi\right\rangle \ud s-\int_{0}^{t}\left\langle \mu_{s},\frac{\left(
\sigma_{2}^{\epsilon}\right)  ^{2}}{2}\partial_{v}^{2}\phi\right\rangle
\ud s\right\vert \wedge1
\end{equation}
where
\begin{equation*}
b\left(  \mu_{s}\right)  \left(  x,v\right)  = \lambda_2\left(  v\right)
+\left\langle \mu_{s},g_{2}\left(  x,v,\cdot,\cdot\right)  \right\rangle \ .
\end{equation*}
\begin{lemma}
\label{lemma Phi}For every ${\phi\in {\cal T}}$ and every $\mu^0\in\mathrm{Pr}(D\times\bT)$, the bounded Borel measurable functional $\Phi^{\mu^0}_{\phi}:\rC\rightarrow\bR $ is continuous at every point of $\widetilde{X}
$. 

Therefore, if $\left\{  Q^{N}\right\}  _{N\in\mathbb{N}}$ and $Q$ are
probability measures on $\rC$ such that { $Q^{N}\rightarrow Q$ weakly} and $Q\left(
\widetilde{X}\right)  =1$, then
\[
\int_{\rC}\Phi^{\mu^0}_{\phi}\ud Q^{N}\rightarrow\int_{\rC}\Phi^{\mu^0}_{\phi}\ud Q\; \text{ as }N\to\infty\ .
\]

\end{lemma}
\begin{proof} Since in the definition of $\Phi^{\mu^0}_{\phi}$ we can consider the supremum over rational numbers in $[0,T]$, to prove the measurability of  $\Phi^{\mu^0}_{\phi}$ we can fix $t\in[0,T]$ and study  separately the measurability of three functionals:
\begin{gather*}
\Phi_1 (\mu)= \left\langle \mu_{t}, \phi\right\rangle  -\left\langle \mu_{0},\phi\right\rangle, \\
\Phi_2 (\mu) = \int_{0}^{t}\left\langle \mu_{s},b\left(  \mu_{s}\right)  \partial_{v}\phi\right\rangle \ud s, \\
\Phi_3 (\mu)= \int_{0}^{t}\left\langle \mu_{s},\frac{\left(\sigma_{2}^{\epsilon}\right)  ^{2}}{2}\partial_{v}^{2}\phi\right\rangle \ud s,
\end{gather*}
for $\mu \in \rC$. Note that $\Phi_1$ and $\Phi_3$ are even continuous mappings on $\rC$. Concerning the measurability of $\Phi_2$ we first approximate pointwise the functions $\lambda_2$ and $g_2$ by regular functions $\lambda_2^n$ and $g_2^n$ (indeed $\lambda_2(\cdot)$ and $g_2(x,\cdot, y, \cdot )$ have only simple discontinuities) and then consider the corresponding functions $b^n$ given by 
\begin{equation} \label{dtt} 
b^n\left(  \mu_{s}\right)  \left(  x,v\right)  =\lambda_2^n\left(  v\right)+\left\langle \mu_{s},g_{2}^n \left(  x,v,\cdot,\cdot\right)  \right\rangle.
\end{equation}
It is not difficult to prove that for each $n$ the functional $\Phi_2^n: \rC \to \bR$, 
$$
\Phi_2^n(\mu) =  
\int_{0}^{t}\left\langle \mu_{s},b^n\left(  \mu_{s}\right)  \partial
_{v}\phi\right\rangle \ud s
$$
is continuous on $\rC$. By the dominated convergence theorem we deduce that $\Phi_2^n(\mu) \to \Phi_2(\mu)$ as $n \to \infty$, for any $\mu \in \rC$. This shows that also $\Phi_2$ is measurable.\\

Let now $\mu\in\widetilde{X}$ and $\mu^{n}\in \rC$ be given with $\mu^{n}%
\rightarrow\mu$ in $\rC$. This implies {$\mu_{t}^{n}\rightarrow\mu_{t}$ in weak sense,}
hence $\left\langle \mu_{t}^{n},\phi\right\rangle \rightarrow\left\langle
\mu_{t},\phi\right\rangle $, uniformly in $t\in[0,T]$. The convergence of $\left\langle \mu_{s}%
^{n},\frac{\left(  \sigma_{2}^{\epsilon}\right)  ^{2}}{2}\partial_{v}^{2}%
\phi\right\rangle $ to $\left\langle \mu_{s},\frac{\left(  \sigma
_{2}^{\epsilon}\right)  ^{2}}{2}\partial_{v}^{2}\phi\right\rangle $ for every
$s\in\left[  0,T\right]  $ is similar and, by Lebesgue dominated convergence
theorem, the last integral in the definition of $\Phi^{\mu^0}_{\phi}$ converges, uniformly in $t\in[0,T]$. It
remains to prove that the first integral converges. Again by Lebesgue
dominated convergence theorem, the problem is reduced to prove that, for a.e.
$s\in\left[  0,T\right]  $,
\[
\left\langle \mu_{s}^{n},b\left(  \mu_{s}^{n}\right)  \partial_{v}%
\phi\right\rangle \rightarrow\left\langle \mu_{s},b\left(  \mu_{s}\right)
\partial_{v}\phi\right\rangle .
\]
This is more difficult since $\lambda_{2}$ and $g_{2}$ contain
discontinuities. Since $\mu\in\widetilde{X}$, we know that $\pi_{v}\mu_{s}\ll\rL_\bT$ for
a.e. $s\in\left[  0,T\right]  $, thus in the sequel we restrict to such values of
$s$. 

Let us first explain why
\begin{equation}
\left\langle \mu_{s}^{n},\lambda_{2}\partial_{v}\phi\right\rangle
\rightarrow\left\langle \mu_{s},\lambda_{2}\partial_{v}\phi\right\rangle
.\label{first convergence}%
\end{equation}
The function $\left(  x,v\right)  \mapsto\lambda_{2}\left(  v\right)
\partial_{v}\phi\left(  x,v\right)  $ is bounded;\ and it is continuous except
on the set
\[
S=D\times\left\{0,1\right\}\subset D\times\bT  .
\]

These sets are exceptional for the measure $\mu_{s}$:%
\begin{align*}
\int_{S}\mu_{s}\left(  \ud x,\ud v\right)  &=\int_{D\times\{v=0\}  }\mu_{s}\left(  \ud x,\ud v\right)+\int_{D\times\{v=1\}}\mu_s\left(\ud x,\ud v\right)\\
 &=\int_{\{  v=0\}\cup\{v=1\}}\left(\pi_v\mu_s\right)\left(\ud v\right)=0\ .
\end{align*}

Now, the following fact is known: if a sequence of probability measures
$\rho_{n}$ on a Polish space $Y$ converges weakly to a probability measures
$\rho$ and $f:Y\rightarrow\bR $ is a bounded Borel measurable function,
continuous on a set $\widetilde{Y}\subset Y$ with $\rho\left(  \widetilde{Y}%
\right)  =1$, then $\int_{Y}fd\rho_{n}\rightarrow\int_{Y}fd\rho$. The proof is
easy using Skorohod representation theorem. We apply this fact with
$Y=D\times\bT $, $\rho_{n}=\mu_{s}^{n}$, $\rho=\mu_{s}$, $\widetilde{Y}%
=S^{c}$, $f=\lambda_{2}\partial_{v}\phi$ and deduce (\ref{first convergence}). 

Finally, let us explain why
\begin{equation}
\label{sww}
\left\langle \mu_{s}^{n},\left\langle \mu_{s}^{n},g_{2}\left(  x,v,\cdot
,\cdot\right)  \right\rangle \partial_{v}\phi\right\rangle -\left\langle
\mu_{s},\left\langle \mu_{s},g_{2}\left(  x,v,\cdot,\cdot\right)
\right\rangle \partial_{v}\phi\right\rangle \rightarrow0.
\end{equation}
The previous difference can be rewritten as the sum of two terms:%
\[
\left\langle \mu_{s}^{n},\left\langle \left(  \mu_{s}^{n}-\mu_{s}\right)
,g_{2}\left(  x,v,\cdot,\cdot\right)  \right\rangle \partial_{v}%
\phi\right\rangle
\]
and
\[
\left\langle \left(  \mu_{s}^{n}-\mu_{s}\right)  ,\left\langle \mu_{s}%
,g_{2}\left(  x,v,\cdot,\cdot\right)  \right\rangle \partial_{v}%
\phi\right\rangle .
\]
The convergence to zero of the second term is similar to
(\ref{first convergence}), because the function
\begin{align*}
D\times\bT\ni\left(  x,v\right)    & \mapsto\left\langle \mu_{s},g_{2}\left(
x,v,\cdot,\cdot\right)  \right\rangle =\int_{D\times\bT} g_{2}\left(  x,v,x^{\prime
},v^{\prime}\right)  \mu_{s}\left(  \ud x^{\prime},\ud v^{\prime}\right)  \\
& =\ind_{\left[  0,1\right]  }\left(  v\right)  \int_{D\times\bT}\theta\left(  x,x^{\prime}\right)  \ind_{[  1,1+\delta]  }\left(  v^{\prime}\right)  \mu_{s}\left(  \ud x^{\prime},\ud v^{\prime}\right)
\end{align*}
is continuous on $S^{c}$. To treat the first term in the sum, we first fix {$\tau >0$}. By the weak convergence, we know that $(\mu_s^n)$ is tight and so  there exists a compact set $K_{\tau} \subset D $ such that 
$$
\mu_s^n ((K_{\tau} \times \bT)^c) < \tau,\;\; \mu_s ((K_{\tau}\times \bT)^c) < \tau,\;\;\; n \ge 1. 
$$
We  have 
\begin{align*}
\left\vert \left\langle \mu_{s}^{n},\left\langle 
\left(  \mu_{s}^{n}-\mu
_{s}\right)  ,g_{2}\left(  x,v,\cdot,\cdot\right)  \right\rangle \partial
_{v}\phi\right\rangle \right \vert  
\leq
 \tau  \Vert \partial_{v}  \phi \Vert _{\infty}
 \Vert \theta  \Vert _{\infty}
\\ 
+ \Vert \partial_{v}  \phi \Vert _{\infty} \int_{ K_{\tau} \times \bT 
}\left\vert \left\langle \left(  \mu_{s}^{n}-\mu_{s}\right)  ,g_{2}\left(
x,v,\cdot,\cdot\right)  \right\rangle \right\vert \mu_s^n (dx, dv) .
\end{align*}
Now, for any $(x,v) \in K_{\tau} \times \bT$, $n \ge 1$,
\begin{align*}
 |\left\langle (\mu_s^n - \mu_{s}),g_{2}\left(
x,v,\cdot,\cdot\right)  \right\rangle | 
\le \left\vert
g_{n}\left(  x\right)  -g\left(  x\right)  \right\vert
\end{align*}
where 
\begin{align*}
g_{n}\left(  x\right)    & :=\int_{D\times\bT}\theta\left(  x, x^{\prime}\right)
\ind_{\left[  1,1+\delta\right]  }\left(  v^{\prime}\right)  \mu_{s}%
^{n}\left(  \ud x^{\prime},\ud v^{\prime}\right),  \\
g\left(  x\right)    & :=\int_{D\times\bT}\theta\left(  x,x^{\prime}\right)  \ind_{\left[
1,1+\delta\right]  }\left(  v^{\prime}\right)  \mu_{s}\left(
\ud x^{\prime},\ud v^{\prime}\right)  .
\end{align*}
To check \eqref{sww} we  have to prove that $g_{n}\rightarrow g$ uniformly on $K_{\epsilon}$. We know it
converges pointwise, by the same argument used above for
(\ref{first convergence}), because the function $\bT\ni v^{\prime}\mapsto \ind_{\left[
1,1+\delta\right]  }\left(  v^{\prime}\right)  $ is continuous apart in $v'=1$ and $v'=1+\delta$. Uniform convergence then follows from the fact that the family
$\left\{  g_{n}\right\}  $ is equi-bounded and equi-uniformly continuous; the last fact is
 a consequence of the assumption that $\theta$ is uniformly continuous on $D \times D$. 

This completes the proof of the first claim of the lemma. The second claim is
a simple consequence using the convergence criterion recalled above in this
proof, applied with $Y=X$, $\widetilde{Y}=\widetilde{X}$, $\rho_{n}=Q^{n}$,
$\rho=Q$, $f=\Phi^{\mu^0}_{\phi}$.
\end{proof}
\begin{lemma}
\label{lemma X tilde}Recall that $\bQ^{N}$ are the laws on $\rC$ of the empirical process
$S^{N}$. If $\bQ$ is a weak limit point of any subsequence of $\left\{
\bQ^{N}\right\}  $ then
\[
\bQ\left(  \widetilde{X}\right)  =1\text{.}%
\]

\end{lemma}

\begin{proof}
\textbf{Step 1}. We have already proved not only tightness of the
family $\left\{  \bQ^{N}\right\}  $ in $\rC$ (see Section \ref{section:Q_tight}) but also tightness of the family of
laws of $u^{N}$ in $\sH:=L^{2}\left(  \left[  0,T\right]  \times\bT\right)  $, where $u_{t}^{N}\left(  v\right)  =\int_{D\times
\bT}\gamma_{N}\left(  v-v^{\prime}\right)  S_{t}^{N}\left(
\ud x^{\prime},\ud v^{\prime}\right)  $ (see the end of Section \ref{subsection:estimates}). Consider the pair $\left(  S^{N}%
,u^{N}\right)  $ with values in $\rC\times \sH$; their laws $\rho^{n}%
=\mathcal{L}\left(  S^{N},u^{N}\right)  $ form a tight family in $\rC\times \sH$.
Given a weak limit point $\bQ$ of $\left\{  \bQ^{N}\right\}  $, there is thus a
subsequence $N_{k}$ such that $\rho^{N_{k}}$ converges weakly to a probability
measure $\rho$ on $\rC\times \sH$, with marginal $\bQ$ on $\rC$. 

By the Skorohod embedding theorem there exists a new probability space $\left(
\widehat{\Omega},\widehat{\sF},\widehat{\bP}\right)  $, $\rC\times \sH$-valued random variables
$\left(  \widehat{S}^{N_{k}},\widehat{u}^{N_{k}}\right)  $ and $\left(
\widehat{S},\widehat{u}\right)  $, with laws $\rho^{N_{k}}$ and $\rho$
respectively, such that $\left(  \widehat{S}^{N_{k}},\widehat{u}^{N_{k}%
}\right)  \rightarrow\left(  \widehat{S},\widehat{u}\right)  $ in the strong
topology of $\rC\times \sH$, with $\widehat{\bP}$-probability one. Notice that the
law of $\widehat{S}$ is $\bQ$. 

\textbf{Step 2}. Let $\mathcal{D}$ be a countable dense family in $C\left(
\bT\right)  $. Let $\rL_{\left[  0,T\right]  }$ be the Lebesgue
measure on $\left[  0,T\right]  $. We claim that, given $\phi\in C\left(
\bT\right)  $, $\left(  \widehat{\bP}\otimes \rL_{\left[
0,T\right]  }\right)  $-almost everywhere,
\begin{equation}
\int_{\bT}\phi\left(  v\right)  \left(  \pi_{v}\widehat{S}%
_{t}\right)  \left(  \ud v\right)  =\int_{\bT}\phi\left(  v\right)
\widehat{u}_{t}\left(  v\right)  \ud v.\label{have density}%
\end{equation}
To prove this claim, let us start from the definition
\begin{equation}
u_{t}^{N_{k}}\left(  v\right)  :=\int_{D\times\bT}\gamma_{N_k}\left(  v-v^{\prime}\right)  S_{t}^{N_{k}}\left(  \ud x^{\prime},\ud v^{\prime}\right), \;\; t \in [0,T],\; v \in \bT.
\label{relation u s}%
\end{equation}
Note that this implies that 
$$
0= \bE \int_0^T dt \int_{\bT} \Big |u_{t}^{N_{k}}\left(  v\right)  - \int_{D\times\bT}\gamma_{N_k}\left(  v-v^{\prime
}\right)  S_{t}^{N_{k}}\left(  \ud x^{\prime},\ud v^{\prime}\right) \Big |^2 dv 
$$
$$
=  \widehat{\bE} \int_0^T dt \int_{\bT} \Big | \widehat u_{t}^{N_{k}}\left(  v\right)  - \int_{D\times\bT}\gamma_{N_k}\left(  v-v^{\prime
}\right)  \widehat S_{t}^{N_{k}}\left(  \ud x^{\prime},\ud v^{\prime}\right) \Big |^2 dv. 
$$
It follows that 
given $N_{k}$, with $\widehat{\bP}$-probability one,
\[
\widehat{u}_{t}^{N_{k}}\left(  v\right)  =\int_{D\times\bT%
}\gamma_{N_{k}}\left(  v-v^{\prime}\right)  \widehat{S}_{t}^{N_{k}}\left(
\ud x^{\prime},\ud v^{\prime}\right)
\]
as an identity in $\sH$. 
 We can also say that 
$\left(  \widehat{\bP}\otimes \rL_{\left[
0,T\right]  }\right)  $-almost everywhere, for any $k \in \bN$
the previous identity holds in $L^2(\bT)$.
Therefore given $\phi\in\mathcal{D}$, we have
\[
 \int_{\bT }\phi\left(  v\right)  \widehat{u}_{t}^{N_{k}}\left(
v\right)  \ud v=  \int_{D\times\bT }\left(  \int_{\bT 
}\gamma_{N_{k}}\left(  v-v^{\prime}\right)  \phi\left(  v\right)  \ud v\right)
\widehat{S}_{t}^{N_{k}}\left(  \ud x^{\prime},\ud v^{\prime}\right).
\]
Up to passing to a subsequence $N_{k}^{\prime}$ we can assume that
$\widehat{u}^{N_{k}^{\prime}}$ converges to $\widehat{u}$ in the strong
topology of $L^{2}\left(  \bT \right)  $, for $\left(
\widehat{\bP}\otimes \rL_{\left[  0,T\right]  }\right)  $-a.e. $\left(
\widehat{\omega},t\right)  \in\widehat{\Omega}\times\left[  0,T\right]  $.
Therefore $\int_{\bT }\phi\left(  v\right)  \widehat{u}_{t}^{N^\prime_{k}%
}\left(  v\right)  \ud v$ converges to $\int_{\bT }\phi\left(
v\right)  \widehat{u}_{t}\left(  v\right)  dv$ for $\left(  \widehat{\bP}\otimes
\rL_{\left[  0,T\right]  }\right)  $-a.e. $\left(  \widehat{\omega},t\right)
\in\widehat{\Omega}\times\left[  0,T\right]  $. And, for $\widehat{\bP}$-a.e.
$\widehat{\omega}\in\widehat{\Omega}$, for every $t\in\left[  0,T\right]  $,
$\widehat{S}_{t}^{N_{k}^{\prime}}$ converges in the weak topology of
probability measures to $\widehat{S}_{t}$, so
\begin{align*}
\lim_{k\rightarrow\infty}\int_{D\times\bT }\left(
\int_{\bT }\gamma_{N_{k}^{\prime}}\left(  v-v^{\prime}\right)
\phi\left(  v\right)  \ud v\right)  \widehat{S}_{t}^{N_{k}^{\prime}}\left(
\ud x^{\prime},\ud v^{\prime}\right)    & =\int_{D\times\bT }%
\phi\left(  v^{\prime}\right)  \widehat{S}_{t}\left(  \ud x^{\prime},\ud v^{\prime
}\right)  \\
& =\int_{\bT }\phi\left(  v\right)  \left(  \pi_{v}\widehat{S}%
_{t}\right)  \left(  \ud v\right)
\end{align*}
where we have used the property%
\[
\lim_{k\rightarrow\infty}\int_{\bT }\gamma_{N_{k}^{\prime}}\left(
v-v^{\prime}\right)  \phi\left(  v\right)  \ud v=\phi\left(  v^{\prime}\right), \; \text{ uniformly in }v^\prime\in\bT\ ,
\]
because $\phi$ is continuous { and $\gamma_N$ are mollifiers on the torus}. This proves (\ref{have density}).

\textbf{Step 3}. Since $\mathcal{D}$ is
countable, property (\ref{have density}) holds true uniformly in $\phi
\in\mathcal{D}$. Hence we can say that, for $\left(  \widehat{\bP}\otimes
\rL_{\left[  0,T\right]  }\right)  $-a.e. $\left(  \widehat{\omega},t\right)
\in\widehat{\Omega}\times\left[  0,T\right]  $, $\pi_{v}\widehat{S}_{t}\left(
\widehat{\omega}\right)  \ll\rL_{\bT }$ with density in $L^2\left(\bT\right)$. This implies that for
$\widehat{\bP}$-a.e. $\widehat{\omega}\in\widehat{\Omega}$, we have the property
that $\pi_{v}\widehat{S}_{t}\left(  \widehat{\omega}\right)  \ll\rL_{\bT}$ with density in $L^2\left(\bT\right)$ for $\rL_{\left[  0,T\right]  }$-a.e. $t\in\left[  0,T\right]  $. Hence
$\widehat{\bP}\left(  \widehat{\omega}\in\widehat{\Omega}:\widehat{S}\left(
\widehat{\omega}\right)  \in\widetilde{X}\right)  =1$, which implies $\bQ\left(
\widetilde{X}\right)  =1$ (recall that the law of $\widehat{S}$ is $\bQ$). The
proof is complete.
\end{proof}

Let us eventually study the nonlinear Fokker-Planck equation (\ref{eq:FP2}), that is
\begin{equation*}
\partial_{t}\mu_{t}+\partial_{v}\left(\mu_t b\left(  \mu_{t}\right)\right)
=\partial_{v}^{2}\left(  \frac{\left(  \sigma_{2}^{\epsilon}\right)  ^{2}}%
{2}\mu_{t}\right),
\end{equation*}
with initial condition $\mu^{0}$, where $b\left(  \mu_{t}\right)  $ is given by \eqref{eq:notation_b}. 
\begin{theorem}
\label{thm:FP}
Let $\mu^0=\nu\times\tilde\rho_0\rL_\bT 
$ with $\nu \in \mathrm{Pr}_{1}(D)$ and $\tilde \rho \in L^2(0,2)$ (cf. Section 1.1).
Then:
\begin{enumerate}[label=\roman{*}$)$]
\item the nonlinear Fokker-Planck equation \eqref{eq:FP2}, with initial condition
$\mu^{0}$, has one and only one weak measure-valued solution $\mu \in \rC$; this
measure belongs to the space $\widetilde{X}$ (see \eqref{eq:xx});
\item 
let $\bQ^{N}$ be the laws on $\rC$ of the empirical process $S^{N}$: then
$\bQ^{N}$ converges weakly to $\delta_{\mu}$. Moreover, $S^{N}$ converges in
probability to $\mu$, in the topology of $\rC$. 
\end{enumerate}
\end{theorem}
\begin{proof} Let us consider the sequence $(\gamma_N)$ of Section \ref{subsection:estimates}.
 Note that 
 the $L^2$ norm of
\begin{equation*}
  u_0^N(v):=\int_{D\times\bT}\gamma_N(v-v^\prime)\mu_0(\ud x^\prime,\ud v^\prime) = \int_{\bT}\gamma_N(v-v^\prime)
   (\pi_{v}\mu_0)(\ud v^\prime)
\end{equation*}
is bounded uniformly in $N$.

From Section \ref{section:Q_tight} we know that the family $\left\{  \bQ^{N}\right\}  $ is tight.
Let $\left\{  \bQ^{N_{k}}\right\}  $ be a weakly convergence subsequence, with
limit $\bQ$. We are going to prove below that $\bQ$ is supported by the set of weak
measure-valued solutions of equation (\ref{eq:FP2}), with initial condition
$\mu^{0}$. This implies existence of at least one such solution. Uniqueness
has been proved in Section \ref{section:uniqueness}; recalling Lemma \ref{lemma X tilde}, we then
immediately have claim (i). Moreover, we also have $\bQ=\delta_{\mu}$, by the
uniqueness of weak measure-valued solutions; therefore, since any weak limit
point of $\left\{  \bQ^{N}\right\}  $ is the same measure $\delta_{\mu}$, it
follows that the full sequence $\left\{  \bQ^{N}\right\}  $ converges weakly to
$\delta_{\mu}$. Since $S^{N}$ converges in law to a constant, it also
converges in probability (in the topology of $\rC$). 

It remains to prove the claim made above that $\bQ$ is supported by the set of
weak measure-valued solutions of equation (\ref{eq:FP2}) with initial
condition $\mu^{0}$, i.e., the following set (cf. \eqref{eq:dff})
\[ \Sigma = 
\left\{  \mu\in \rC:\Phi^{\mu^0}_{\phi}\left(  \mu\right)  =0\text{ for all }\phi
\in\mathcal{T}\right\}  
\]
is a Borel subset of ${\cal C}$ and $\bQ(\Sigma)=1$. Arguing as in Remark \ref{maa} it is not difficult to show that given $\phi \in {\cal T}$ there exists a sequence $(\phi_n) \subset C_c^{2}(D \times \bT)$ (i.e. $\phi_n$ is a $C^2$ function with compact support) such that 
$\| \phi_n\|_{\infty}$ $+ \| \partial_v \phi_n  \|_{\infty} $ $+\| \partial_{vv}^2 \phi_n  \|_{\infty} $ 
 $ \le M$ (for some $M>0$ independent of $n$)
 and $\phi_n(z) \to \phi(z)$, $\partial_v \phi_n (z) \to 
 \partial_v \phi(z)$, $\partial_{vv}^2 \phi_n (z) \to 
 \partial_{vv}^2 \phi(z)$ as $n \to \infty$, for any $z = (x, v) \in D \times \bT$. Hence by the dominated convergence theorem we have
$$
\Sigma = \left\{  \mu\in \rC:\Phi^{\mu^0}_{\phi}\left(  \mu\right)  =0\text{ for all }\phi
\in C_c^{2} (D \times \bT)\right\}.
$$ 
Moreover, since there exists   a countable
set $H_0 \subset C^{2}_c(D \times \bT)$ such 
 that for any $\phi \in C^{2}_c(D \times \bT) $
we can find  a sequence $(\phi_k) \subset H_0$ satisfying
$$
 \lim_{k \to \infty} ( \| \phi - \phi_k \|_{\infty} +   \| \partial_v \phi_k -  \partial_v \phi  \|_{\infty} +   \| \partial_{vv}^2 \phi_k -  \partial_{vv}^2 \phi  \|_{\infty}) =0,
$$
we obtain that $\Sigma = \left\{  \mu\in \rC:\Phi_{\phi}\left(  \mu\right)  =0\text{ for all }\phi
\in H_0 \right\}$ which is a Borel subset of ${\cal C}$.
To finish the proof we need to prove that
\[
\bQ\left(  \mu\in \rC:\Phi^{\mu^0}_{\phi}\left(  \mu\right)  =0\right)  =1 \ ,
\]
for every $\phi\in H_0$. Since $\Phi^{\mu^0}_{\phi}\geq0$, it is equivalent to
prove $
\int_{\rC}\Phi^{\mu^0}_{\phi}\ud\bQ=0,
$
for every $\phi\in H_0$. Due to Lemma \ref{lemma Phi}, it is enough
to show
\begin{equation}
\lim_{N\rightarrow\infty}\int_{\rC}\Phi^{\mu^0}_{\phi}\ud\bQ^{N}=0\label{lim to be proved}%
\end{equation}
for every $\phi\in H_0$. Using identity (\ref{eq:Ito_SN}) we have%
\begin{align*}
& \int_{\rC}\Phi^{\mu^0}_{\phi}\ud\bQ^{N}\\
& =\bE\left[  \sup_{t\in\left[  0,T\right]  }\left\vert \left\langle
S_{t}^{N},\phi\right\rangle -\left\langle \mu^{0},\phi\right\rangle -\int%
_{0}^{t}\left\langle S_{s}^{N},b\left(  S_{s}^{N}\right)  \partial_{v}%
\phi\right\rangle \ud s-\int_{0}^{t}\left\langle S_{s}^{N},\frac{\left(
\sigma_{2}^{\epsilon}\right)  ^{2}}{2}\partial_{v}^{2}\phi\right\rangle
\ud s\right\vert \wedge1\right]  \\
& =\bE\left[  \sup_{t\in\left[  0,T\right]  }\left\vert M_{t}^{N,\phi
}\right\vert \wedge1\right]  \leq\bE\left[  \sup_{t\in\left[
0,T\right]  }\left\vert M_{t}^{N,\phi}\right\vert ^{2}\right]  ^{1/2}%
\end{align*}
where%
\[
M_{t}^{N,\phi}=\int_{0}^{t}\frac{1}{N}\sum_{i=1}^{N}\sigma_{2}^{\epsilon
}\left(  V_{s}^{i,N}\right)  \partial_{v}\phi\left(  X_0^{i},V_{s}^{i,N}\right)
\text{d}B_{s}^{i}.
\]
Therefore, from Doob's inequality and the boundedness of $\sigma_{2}%
^{\epsilon}$ and $\partial_{v}\phi$, for some constants generically denoted by
$C>0$, we have
\[
\left(  \int_{\rC}\Phi^{\mu^0}_{\phi}\ud\bQ^{N}\right)  ^{2}\leq C\int_{0}^{T}\frac{1}%
{N^{2}}\sum_{i=1}^{N}\bE\left[  \left\vert \sigma_{2}^{\epsilon}\left(
V_{s}^{i,N}\right)  \partial_{v}\phi\left(  X_0^{i},V_{s}^{i,N}\right)
\right\vert ^{2}\right]  \ud s\leq\frac{C}{N}%
\]
which implies (\ref{lim to be proved}) and completes the proof. 
\end{proof}

\section{The McKean-Vlasov SDE}
\label{section:vlasov}
Similarly to what is done in Section \ref{section:SDE}, let $B$ be a standard real Brownian motion defined on $(\Omega,\sF,\bP)$, let $\eta$ be a $\bT$-valued random variable independent of $B$ with density $\tilde\rho_0\in L^2(0,2)$ and denote by $\sG^0_t$ the augmentation of $\sigma(B_s,0\leq s\leq t)\vee\sigma(\eta)$ with the $\bP$-null sets.
Let then $\xi$ be a random variable with values in $D$ and law $\nu$, having finite first moment, independent of $\sG^0_t$ for every $t$. Finally denote by $\sG_t$ the completion of $\sG^0_t\vee\sigma(\xi)$.\\
Let us consider the so-called McKean-Vlasov SDE associated with the system (\ref{eq:system3})
\begin{equation}
  \label{eq:MKV}
  \begin{cases}
\ud V_t=\lambda_2(V_t)\ud t+\int_{D\times\bT} g_2(\xi,V_t,y,w)\mu_t(\ud y,\ud w)\ud t+\sigma_2^\epsilon(V_t)\ud B_t \\
\ud X_t=0\\
(X_0,V_0)=(\xi,\eta)\\
\mu_t=\sL(X_t,V_t)\ .\\
\end{cases}
\end{equation}
In analogy with Definition \ref{def:strongsolution}, we say that a strong solution to equation (\ref{eq:MKV}) is a family of continuous $\sG^0_t$-adapted processes $(V^x_t,\mu^x_t)$, $x\in D,$ with values in $\bR\times\mathrm{Pr}_1(D\times\bT)$ such that the mapping: $(x,\omega)\mapsto(V^x_\cdot (\omega),\mu^x_\cdot (\omega))$ is measurable $(D\times\Omega,\sB(D)\times\sF)\to(\sS\times\rC,\sB(\sS)\times\sB(\rC))$ and for $\nu$-almost every $x\in D$ the process $(V^x,\mu^x)$ satisfies (\ref{eq:MKV}) for $\xi=x$. If $(V^x,\mu^x)$ is a strong solution then $(\xi,V^{\xi}_t,\mu^{\xi}_t=\sL(\xi,V^{\xi}_t))$ is a well-defined continuous $\sG_t$-adapted process with values in $D\times\bR\times\mathrm{Pr}_1(D\times\bT)$ which satisfies equations (\ref{eq:MKV}) $\bP$-almost surely for any $t\in[0,T]$.
\begin{theorem}
\label{thm:MKV}  
There exist a unique strong solution $(V,\mu)$ to equation (\ref{eq:MKV}).
\end{theorem}
\begin{proof}
 Fix a measure $\tilde\mu_\cdot\in \rC$ and set
  \begin{equation*}
    b(\tilde\mu_t)(x,v):=\lambda_2(v)+\int_{\bR^3}\int_\bR g_2(x,v,y,w)\tilde\mu_t(\ud w,\ud y)\ .
  \end{equation*}
By the same arguments as in Section \ref{section:SDE} it can be proved that there exist a unique strong solution $(V^{\tilde\mu,x})_x$ to the SDE
  \begin{equation}
    \label{eq:MKVthm}
    \begin{cases}
    \ud V^{\tilde\mu,x}_t=b(\tilde\mu_t)(x,V_t^{\tilde\mu,x})\ud t+\sigma_2^\epsilon(V^{\tilde\mu,x}_t)\ud B_t\ ,\\
    V^{\tilde\mu,x}_0=\eta\ .\\
  \end{cases}
\end{equation}
Hence for $X_0$ as in (\ref{eq:MKV}) we have that $V^{\tilde\mu}:=V^{\tilde\mu,X_0}$ satisfies
  \begin{equation}
    \label{eq:MKVthm2}
    \begin{cases}
    \ud V^{\tilde\mu}_t=b(\tilde\mu_t)(X_0,V_t^{\tilde\mu})\ud t+\sigma_2^\epsilon(V^{\tilde\mu}_t)\ud B_t\ ,\\
    V^{\tilde\mu}_0=\eta\ .
  \end{cases}
\end{equation}
Now choose as $\tilde\mu$ the unique weak solution in $\rC$ to the nonlinear Fokker-Planck PDE (\ref{eq:FP2}) with initial condition $\tilde\mu^0=\nu\times\tilde\rho_0\rL_\bT$; given the associated process $V^{\tilde\mu}$ as above, denote by $\mu_t$ the law of the vector $(\xi ,V^{\tilde\mu}_t)$. Then $\mu$ is a solution in $\rC$ to the linear PDE
\begin{equation*}
\begin{cases}
   \frac{\partial}{\partial t}\mu_t=\frac{1}{2}\frac{\partial^2}{\partial v^2}\left(\left(\sigma_2^\epsilon\right)^2\mu_t\right)-\frac{\partial}{\partial v}\big(\mu_tb(\tilde\mu_t)\big)\\
\mu_0=\tilde\mu^0\ .
\end{cases}
\end{equation*}
Since there is uniqueness of measure-valued solution to the latter (the proof being a simplified version of that of uniqueness in the nonlinear case; see Theorem \ref{thm:xd}), and clearly also $\tilde\mu$ is a solution, we necessarily have $\tilde\mu=\mu$ and $(V,\mu)$ is a strong solution to (\ref{eq:MKV}).

\smallskip
Let now $(\overline{V},\bar\mu)$ be another solution. Then $\bar\mu^x$ solves the Fokker-Planck equation (\ref{eq:FP2}) with initial condition $\delta_x \times \tilde\rho_0\rL_\bT$, for $\nu$-almost every $x\in D$, and therefore $\bar\mu^x=\mu^x$. Finally $V^x_t=\overline{V}^x_t$ a.s. for every $t$ and $\nu$-almost every $x\in D$, since they both satisfy a SDE like (\ref{eq:MKVthm2}) with $\tilde\mu=\mu^x$ for which strong uniqueness holds. Hence $(V,\mu)$ is the unique solution to (\ref{eq:MKV}).
\end{proof}

\appendix
\section{Appendix: Extension of some results}
\label{section:extension}
We state here some further results that are easy generalizations of what we presented so far; we will comment on the proofs when needed.

Notice that in Section \ref{section:SDE} neither the particular form of the coefficients nor the fact that they are $2$-periodic plays any role. Since the cited results we built our proof on apply to multidimensional SDE with bounded and measurable drift, we immediately obtain the following theorem, with proof identical to that of Theorem \ref{thm:SDEE1}. Here, similarly to what done previously, for $k,l,m,n\in\bN$ we consider a $m$-dimensional Brownian motions $\mathbf{W}$ and independent random variables $\Xi=\left(\Xi^j\right)_{j=1,\dots,n}$ and $\Psi=\left(\Psi^j\right)_{j=1,\dots,n}$ (independent from $\mathbf{W}$ as well) with values in $E^n\subseteq\left(\bR^k\right)^n$ and $\left(\bR^l\right)^n$, respectively ($E$ is an open subset of $\bR^k$), all defined on the common probability space $(\Omega,\sF,\bP)$ and with finite first moment. We assume that the $\Xi^j$'s are identically distributed with law $\bm{\nu}$ and the $\Psi^j$'s are identically distributed with absolutely continuous law $\bm{\rho}_0\rL_{\bR^l}$, with $\bm{\rho}_0\in L^2(\bR^l)$. We define the $\sigma$-algebra $\sE=\sigma(\Xi)$, the filtration $(\sA^0_t)_t$ as the augmentation of $\sigma(\mathbf{W}_s,0\leq s\leq t)\vee\sigma(\Phi)$ with the $\bP$-null sets and the filtration $(\sA_t)_t$ as the completion of $\sA_t^0\vee\sE$. Finally we fix $T>0$.
\begin{theorem}
  Let $\mathbf{b}:[0,T]\times E^n\times\bR^{l\times n}\times E\times\bR^l\to\bR^{l\times n}$ and $\bm{\sigma}:[0,T]\times\bR^{l\times n}\to\bR^{l\times n\times m}$ be bounded and Borel measurable functions. Assume that $\bm{\sigma}\bm{\sigma}^\top(t,\cdot)$ is uniformly elliptic and that $\bm{\sigma}(t,\cdot)$ is Lipschitz; both properties have to be satisfied uniformly in $t$. Then there exists a unique strong solution $\mathbf{Y}=\left(Y^j\right)_{j=1,\dots,n}$, in the sense of Definition \ref{def:strongsolution} and Theorem \ref{thm:SDEE1}, to the SDE
  \begin{equation}
  \label{eq:SDEgeneral}
  \begin{cases}
    \ud \mathbf{Y}_t=\left\langle\mathbf{b}\left(t,\Xi,\mathbf{Y}_t,\cdot,\cdot\right),\overline{S}^n_t\right\rangle\ud t+\bm{\sigma}(t,\mathbf{Y}_t)\ud \mathbf{W}_t\ ,\qquad t\in[0,T]\ ,\\
    \mathbf{Y}_0=\Psi
  \end{cases}
  \end{equation}
where
\begin{equation*}
  \overline{S}_t^n=\frac{1}{n}\sum_{j=1}^n\delta_{\left(\Xi^j,Y^j\right)}\ .
\end{equation*}
 \end{theorem}

Finer refinements are possible: for example one can treat the cases when $T=\infty$ and the SDE is to be solved on a domain $U\subset\bR^{l\times n}$, and the assumptions on $\bm{\sigma}$ can be weakened (cf. Remark \ref{rem:sigma}). We refer to \cite{GK96} for details.\\

Also when studying our Fokker-Planck PDE the periodicity of the coefficients plays no particular role, nor does the compactness of the torus $\bT$.\\
We need anyway some more assumptions on the coefficients $\mathbf{b}(t,x,y,x^\prime,y^\prime)$, $t\in [0,T]$, $x\in E^n$, $y\in \bR^{t\times n}$, $x^\prime\in E$, $y^\prime\in \bR^l$, and $\bm{\sigma}(t,y)$ above, namely:
\begin{enumerate}[label=$($E\arabic{*}$)$]
\item\label{E1} $\mathbf{b}$ does not depend on $t$, is bounded and uniformly continuous in $x$ and $x'$ (uniformly in  $y$ and $y'$) and 
the set of discontinuities of the map
  \begin{equation*}
    (y,y')\mapsto \mathbf{b}(x,y,x',y')
  \end{equation*}
has Lebesgue measure $0$ in $\bR^{l\times n}\times \bR^{l}$, for any $x,x' $;
\item\label{E2} $\bm{\sigma}$ does not depend on $t$ and belongs to $C^1_b(\bR^{l \times n}; \bR^{l\times n\times m})$.
\end{enumerate}
One can repeat the arguments of Sections \ref{section:uniqueness}, \ref{section:Q_tight}, \ref{section:PDE_existence} with minor modifications, working in the space $\mathrm{Pr}_1(E\times\bR^l)$ with the $1$-Wasserstein metric, using the Euclidean norm in place of the metric $d_{D\times\bT}$ and choosing all the test functions accordingly. If we solve equation (\ref{eq:SDEgeneral}) above for $\mathbf{Y}$ and consider the empirical measure $\overline{S}^n$, we can define the empirical density $\bar u^n$ as
\begin{equation*}
  \bar u^n_t(y)=\int_{E\times\bR^l}\bar\gamma_n(y-y^\prime)\overline{S}_t^n(\ud x^\prime,\ud y^\prime)
\end{equation*}
where $\left(\bar\gamma_n\right)$ is a family of compactly supported mollifiers in $\bR^l$.\\
We say that $f\in L^2_{\mathrm{loc}}(\bR^l)$ if $f\in L^2(K)$ for every compact set $K\subset \bR^l$. Consider a strictly increasing sequence $(P_j)_{j\in\bN}$ of compact sets in $\bR^l$ such that $\bR^l=\cup_j P_j$; then 
\begin{equation*}
  d(f,g) = \sum_{j \ge 1} \frac{1}{2^j} \frac{\|f-g \|_{L^2(P_j)}}{1 + \| f-g \|_{L^2(P_j)}}
\end{equation*}
is a metric on $L^2_{\mathrm{loc}}(\bR^l)$. Lemmata \ref{lem:51} and \ref{lem:52} apply to $\bar u^n$'s as well thanks to assumption \ref{E2}, and imply tightness of their laws in the space $L^2\left(0,T;L^2_{\mathrm{loc}}\left(\bR^l\right)\right)$ due to a generalized version of Aubin-Lions lemma, which claims that the space
\begin{equation*}
L^{2}\left(  0,T;W^{1,2}\left(  \bR^l \right)  \right)  \cap W^{\alpha
,2}\left(  0,T;H^{-2}\left(  \bR^l \right)  \right)
\end{equation*}
is relatively compact in $L^{2}\left(  0,T;L_{loc}^{2}\left(  \bR^l\right)  \right)  $.\\
Let $\sC:=C\left([0,T];\mathrm{Pr}_1\left(E\times\bR^l\right)\right)$. The space $\widetilde X$ of Subsection \ref{subsection:existence} has to be consequently substituted with
$$
\overset{\,\diamond}{X}=\Big\{\mu\in \sC \; : \; 
 \pi_{v}\mu_{t}\ll\rL_{\bR^l}\; \text{with} \; \frac{\ud\left(\pi_{v}\mu_{t}\right)}{\ud \rL_\bR^l} \in L^2 (\bR^l),\;\text{ for a.e. }t\in\left[  0,T\right]  \Big\}\ .
$$
To show that $\overset{\,\diamond}{X}$ is Borel it is enough to repeat the argument in the proof of Lemma \ref{lem:53} using the density of $C_c(\bR^l)$ in $L^2(\bR^l)$. The functionals $\Phi^{\mu^0}_\phi$ are now continuous on $\overset{\,\diamond}{X}$, because the proof of Lemma \ref{lemma Phi} can be repeated thanks to assumption \ref{E1}. To show that any limit point of the family of laws of the $\overline{S}^n$ gives full measure to $\overset{\,\diamond}{X}$ one can repeat the proof of Lemma \ref{lemma X tilde}, noting that it is enough to check identity (\ref{have density}) for $\phi\in C_b(\bR^l)$.\\
Now we can repeat verbatim the arguments in the proof of Theorem \ref{thm:FP} obtaining:
\begin{theorem}
  \label{thm:FP2}
Let $\zeta^0=\bm{\nu}\times\bm{\rho_0}\rL_{\bR^l}$.
Then:
\begin{enumerate}[label=\roman{*}$)$]
\item the nonlinear Fokker-Planck equation
\begin{equation}
\label{eq:FPbig}
  \partial_{t}\zeta_{t}+\operatorname{div}_{y}\left(\zeta_t \langle\mathbf{b},\zeta_{t}\rangle\right)= \text{Tr}\Big [D_{y}^{2}\left(  \frac{\bm{\sigma  \sigma^T} }{2}\zeta_{t}\right) \Big]
\end{equation}
with initial condition $\zeta^{0}$, has one and only one weak measure-valued solution $\zeta$; this measure belongs to the space $\overset{\,\diamond}{X}$.

\item 
Let $\bQ^{n}$ be the laws on $\sC$ of the empirical process $S^{n}$; then $\bQ^{n}$ converges weakly to $\delta_{\zeta}$. Further  $S^{n}$ converges in probability to $\zeta$, {in the topology of $\sC$.}
\end{enumerate}
\end{theorem}
\noindent The extension fo the well-posedness result for the McKean-Vlasov equation given in Section \ref{section:vlasov} is then straightforward. {For related results on strong well-posedness for McKean-Vlasov equations (without dependence on stochastic parameters) see \cite{MV16} and references therein.}

\bibliographystyle{halpha}

\end{document}